\documentclass[reqno]{amsart}

\title[{}]{Duality  via convolution of $\W$-algebras}

\author{Thomas Creutzig}
\address[T.C.]{Department Mathematik, FAU Erlangen–Nürnberg, Cauerstraße 11, 91058, Erlangen, Germany}
\email{thomas.creutzig@fau.de}

\author{Andrew R. Linshaw}
\address[A.L.]{Department of Mathematics, University of Denver}
\email{andrew.linshaw@du.edu}

\author{Shigenori Nakatsuka}
\address[S.N.]{Department Mathematik, FAU Erlangen–Nürnberg, Cauerstraße 11, 91058, Erlangen, Germany}
\email{shigenori.nakatsuka@fau.de}

\author{Ryo Sato}
\address[R.S.]{Center for General Education, Aichi Institute of Technology,
Yakusa-cho, Toyota 470-0392, Japan}
\email{rsato@aitech.ac.jp}

\usepackage{latexsym}
\usepackage{amsmath}
\usepackage{amsfonts}
\usepackage{amsthm}
\usepackage{amssymb}
\usepackage{amsxtra}
\usepackage{mathrsfs}
\usepackage{eucal}
\usepackage[all]{xy}
\usepackage{color}
\usepackage{braket}
\usepackage{graphics, setspace}
\usepackage{etoolbox, textcomp}
\usepackage{comment}
\usepackage{changepage}
\usepackage{xcolor}
\definecolor{rouge}{rgb}{0.85,0.1,.4}
\definecolor{bleu}{rgb}{0.1,0.2,0.9}
\definecolor{violet}{rgb}{0.7,0,0.8}
\usepackage[colorlinks=true,linkcolor=bleu,urlcolor=violet,citecolor=rouge,breaklinks]{hyperref}
\usepackage{cleveref}
\usepackage{lscape}
\usepackage{caption}
\usepackage{subcaption}

\newtheorem{definition}{Definition}[section]
\newtheorem{proposition}[definition]{Proposition}
\newtheorem{theorem}[definition]{Theorem}
\newtheorem*{maintheorem}{Main Theorem}
\newtheorem{corollary}[definition]{Corollary}
\newtheorem{lemma}[definition]{Lemma}
\newtheorem{remark}[definition]{Remark}

\numberwithin{equation}{section}

\newcommand{\Z}{\mathbb{Z}}

\newcommand{\C}{\mathbb{C}}
\newcommand{\ag}{\mathfrak{a}}
\newcommand{\bg}{\mathfrak{b}}
\newcommand{\End}{\operatorname{End}}
\newcommand{\Hom}{\operatorname{Hom}}
\newcommand{\Span}{\operatorname{Span}}
\newcommand{\Com}{\operatorname{Com}}
\newcommand{\id}{\operatorname{id}}
\newcommand{\Ker}{\operatorname{Ker}}
\newcommand{\Coker}{\operatorname{Coker}}

\newcommand{\tr}{\operatorname{tr}}
\newcommand{\str}{\operatorname{str}}

\newcommand{\pd}{\partial}
\newcommand{\gr}{\operatorname{gr}}

\newcommand{\lbr}[2]{ {[} {#1} {}_\lambda {#2} {]}}
\newcommand{\lm}[1]{\lambda^{(#1)}}
\newcommand{\oo}[1]{o(\lambda^{#1})}
\newcommand{\oom}[1]{o(\mu^{#1})}

\newcommand{\Q}{\mathbb{Q}}
\newcommand{\g}{\mathfrak{g}}
\newcommand{\fd}{\mathfrak{d}}
\newcommand{\gl}{\mathfrak{gl}}
\newcommand{\sll}{\mathfrak{sl}}
\newcommand{\osp}{\mathfrak{osp}}
\newcommand{\symp}{\mathfrak{sp}}
\newcommand{\so}{\mathfrak{so}}

\newcommand{\op}[1]{\breve{#1}}
\newcommand{\tp}{{\scalebox{0.5}{$+$}}}
\newcommand{\tn}{{\scalebox{0.5}{$-$}}}
\newcommand{\tpn}{{\scalebox{0.5}{$\pm$}}}
\newcommand{\tnp}{{\scalebox{0.5}{$\mp$}}}
\newcommand{\A}{\mathfrak{A}^n}

\newcommand{\tran}{\hspace{0mm}^t}

\newcommand{\W}{\mathcal{W}}

\newcommand{\KVA}[1]{\mathfrak{A}^{{#1}}}
\newcommand{\cdo}{\mathcal{D}_{G,k}^{\mathrm{ch}}}
\newcommand{\NO}[1]{:\!#1\!:}
\newcommand{\ssqrt}[1]{\operatorname{\sqrt{\smash[b]{#1}}}}
\newcommand{\rel}[1]{H^{\frac{\infty}{2}+#1}_{\mathrm{rel}}}
\newcommand{\KL}{\mathbf{KL}}
\newcommand{\lp}{\operatorname{L}}
\newcommand{\semi}{\bigwedge\hspace{-0.5mm}^{\frac{\infty}{2}+\bullet}(\g)}
\newcommand{\semiinf}[2]{\bigwedge\hspace{-0.5mm}^{\frac{\infty}{2}+{#1}}(\g)_{#2}}
\newcommand{\relsemi}{\bigwedge\hspace{-0.5mm}^{\frac{\infty}{2}+\bullet}_{\hspace{1mm}\operatorname{rel}}}
\newcommand{\semicoh}[1]{H^{\frac{\infty}{2}+#1}}
\newcommand{\relsemicoh}[1]{H^{\frac{\infty}{2}+#1}_{\operatorname{rel}}}
\newcommand{\weyl}{\mathbb{V}}
\newcommand{\lpp}[1]{\mathrm{L}^{\!\overset{#1}{\hspace{2mm}}\!}}
\newcommand{\odd}[1]{\psi_{\scalebox{0.6}{$#1$}}}
\newcommand{\sign}[1]{\scalebox{0.6}{$#1$}}

\newcommand{\arxiv}[2]{\href{https://arxiv.org/abs/#1}{#2}}

\newcommand\doi[2]{\href{http://dx.doi.org/#1}{#2}}

\begin{document}
\maketitle

\begin{abstract}
Feigin-Frenkel duality is the isomorphism between the principal $\W$-algebras of a simple Lie algebra $\mathfrak{g}$ and its Langlands dual Lie algebra ${}^L\g$. A generalization of this duality to a larger family of $\W$-algebras called {\it hook-type} was recently conjectured by Gaiotto and Rap\v{c}\'ak and proved by the first two authors. It says that the affine cosets of two different hook-type $\W$-(super)algebras are isomorphic. A natural question is whether the duality between the affine cosets can be enhanced to reconstruct one side of the $\W$-algebra from the other. There is a convolution operation that maps a hook-type $\W$-algebra $\W$ to a certain relative semi-infinite cohomology of $\W$ tensored with a suitable kernel VOA. The first two authors conjectured previously that this cohomology is isomorphic to the Feigin-Frenkel dual hook-type $\W$-algebra. Our main result is a proof of this conjecture.
\end{abstract}

\section{Introduction}

Let $\g$ be a simple Lie algebra, or more generally, a basic-classical simple Lie superalgebra, and $V^k(\g)$ be the affine vertex superalgebra associated with $\g$ at level $k$. 
To any even nilpotent element $f$ in $\g$ one constructs from $V^k(\g)$ a corresponding $\W$-superalgebra, denoted by $\W^k(\g, f)$, via the quantum Drinfeld-Sokolov reduction associated with $f$ \cite{FF90b, KRW}. 
When $f$ is principal, the corresponding $\W$-superalgebras are commonly denoted by $\W^k(\g)$. The principal $\W$-superalgebras for Lie algebras $\g$ enjoy Feigin-Frenkel duality \cite{FF1}, that is, the isomorphisms
\[
\W^k(\g) \simeq \W^\ell({}^L\g), \qquad r^\vee(k+h^\vee)(\ell + {}^Lh^\vee) =1.
\]
Here, ${}^L\g$ is the Langlands dual of $\g$, $r^\vee$ the lacing number of $\g$, and $h^\vee$ (resp.\ ${}^Lh^\vee$) the dual Coxeter number of $\g$ (resp.\ ${}^L\g$). 
In the super case, a similar duality is known for $\g = \osp_{1|2n}$ with ${}^L\g=\osp_{1|2n}$ and $r^\vee=4$ \cite{CL2, G}. 
{Feigin-Frenkel duality plays a fundamental role in the quantum geometric Langlands program as the $\W$-algebra $\W^k(\g)$ describes the center of the affine Lie algebra $\widehat{\g}$ at the critical level, relating the center to local systems appearing on the other side, which is deformed for all levels at the same time \cite{Fr0}.}

{The hook-type $\W$-superalgebras are $\W$-superalgebras studied in \cite{CL1,CL2, GR} and will be denoted by 
$$\W^k_{A^\pm}(n,m),\quad \W^k_{B^\pm}(n,m),\quad \W^k_{C^\pm}(n,m),\quad \W^k_{D^\pm}(n,m),\quad \W^k_{O^\pm}(n,m)$$
in this paper where $n\geq1,m\geq0$ for $X=A$ and $n\geq0,m\geq0$ for the other cases, 
see Table \ref{tab: hopok-type Walg-intro}.} (They appear as $\W^k(\g, f_\ag)$ in \cite{CL1,CL2}, but we use a different notation to present our result uniformly for all types.) 
{In type $A$, for example, they can be regarded as a family of $\W$-algebras interpolating the principal $\W$-algebras as $\W^k_{A^+}(n,0)$ and the affine vertex algebras as $\W^k_{A^+}(1,m)$.
The name ``hook-type" comes from the shape of the Young diagrams associated with the partitions for the nilpotent elements (e.g.\ $[n,1^m]$ for $\W^k_{A^+}(n,m)$).}
These $\W$-superalgebras are parameterized by embeddings $\ag\oplus\bg\subset \g$ with nilpotent element taken to be principal inside $\ag$ so that 
$$\g\simeq
\begin{cases}
(\ag\oplus \bg)\oplus \rho_\ag\otimes \rho^\dagger_\bg\oplus \rho^\dagger_\ag\otimes \rho_\bg & \text{if}\ \g=\sll_{N|M},\\
(\ag\oplus \bg)\oplus \rho_\ag\otimes \rho_\bg& \text{if}\ \g=\osp_{N|2M},
\end{cases}$$
holds where $\rho_\ag$ and $\rho^\dagger_\ag$ (resp.\ $\rho_\bg$ and $\rho^\dagger_\bg$) are the natural representation of $\ag$ (resp.\ $\bg$) and its dual. 
Note that $\bg$ is the centralizer of $\ag\subset \g$ in all cases, and $\bg$ is of type $A, B, C, D, O$ with $m$ representing the rank in our notation. It generates a (usually maximal) affine vertex subalgebra $V^{k_\bg}(\bg)$ at a certain level $k_\bg$ in the $\W$-superalgebra.
\renewcommand{\arraystretch}{1.2}
\begin{table}[h]
    \caption{Hook-type $\W$-superalgebras}
    \label{tab: hopok-type Walg-intro}
    \centering
\begin{tabular}{|cccc|}
\hline
Name for $\W^k(\g, f_\ag)$ & $\g$ & $\ag$ & $\bg$  \\ \hline  
$\W^k_{A^+}(n,m)$ & $\sll_{n+m}$ & $\sll_n$ & $\gl_m$  \\ 
$\W^k_{B^+}(n,m)$ & $\so_{2(n+m+1)}$          & $\so_{2n+1}$            & $\so_{2m+1}$    \\
$\W^k_{C^+}(n,m)$ &$\symp_{2(n+m)}$&$\symp_{2n}$ & $\symp_{2m}$ \\ 
$\W^k_{D^+}(n,m)$ & $\so_{2(n+m)+1}$&$\so_{2n+1}$&  $\so_{2m}$ \\ 
$\W^k_{O^+}(n,m)$ & $\osp_{1|2(n+m)}$ & $\symp_{2n}$ & $\osp_{1|2m}$  \\ \hline 
$\W^k_{A^-}(n,m)$ & $\sll_{n+m|m}$ &$\sll_{n+m}$ &$\gl_m$    \\
$\W^k_{B^-}(n,m)$ &$\osp_{2m+1|2(n+m)}$ &$\symp_{2(n+m)}$ &$\so_{2m+1}$ \\
$\W^k_{C^-}(n,m)$ & $\osp_{2(n+m)+1|2m}$     & $\so_{2(n+m)+1}$    &$\symp_{2m}$  \\ 
$\W^k_{D^-}(n,m)$ & $\osp_{2m|2(n+m)}$& $\symp_{2(n+m)}$& $\so_{2m}$ \\
$\W^k_{O^-}(n,m)$ & $\osp_{2(n+m)+2|2m}$ & $\so_{2(n+m)+1}$ & $\osp_{1|2m}$  \\ \hline
\end{tabular}
\end{table}
\renewcommand{\arraystretch}{1}

Davide Gaiotto and Miroslav Rap\v{c}\'ak studied vertex algebras at the two-dimensional corners in four-dimensional quantum field theories \cite{GR}. 
The vertex algebras appearing in their work are called vertex algebras at the corner.
{They are the subalgebras of $\W^k_{X^\pm}(n,m)$ consisting of elements commuting with $V^{k_\bg}(\bg)$:
$$C^k_{X^\pm}(n,m)=\Com(V^{k_\bg}(\bg),\W^k_{X^\pm}(n,m)).$$
The $S_3$-symmetry of the theories led them to conjecture a remarkable family of {\it trialities} among these vertex algebras, i.e., isomorphisms of three such algebras. These isomorphisms fall into two types. One type generalizes Feigin-Frenkel duality, namely 
\begin{align}\label{gen FF duality}
    C^k_{X^+}(n,m)\simeq C^\ell_{Y^-}(n,m)
\end{align}
for the pairs $(X, Y)=(A, A), (B, O), (C, C), (D, D), (O, B)$ and the levels $(k,\ell)$ satisfying the relation in \eqref{duality rel}.}

The other type generalizes the following coset realizations of the principal $\W$-algebras $\W_{X^+}^k(n,0)$ \cite{ACL, CL2}
\begin{align*}
{\W_{X^+}^k(n,0)} &\simeq \Com\left(  V^\ell(\g), V^{\ell-1}(\g) \otimes L_1(\g)\right)\quad (X=A,D),\quad  & \frac{1}{k+h^\vee_1} + \frac{1}{\ell + h^\vee_2} = 1,  \\ 
{\W_{C^+}^k(n,0)} &\simeq \Com\left(  V^\ell(\symp_{2n}), V^{\ell}(\osp_{1|2n}) \right), & \frac{1}{k+h^\vee_1} + \frac{1}{\ell + h^\vee_2} = 2, \\
{\W_{O^+}^k(n,0) }&\simeq \Com\left(  V^\ell(\so_{2n+1}), V^{\ell-1}(\so_{2n+1}) \otimes F^{2n+1} \right), & \frac{1}{k+h^\vee_2} + \frac{2}{\ell + h^\vee_2} = 2.
\end{align*}
Here, $L_1(\g)$ in the first isomorphism is the simple quotient of $V^{1}(\g)$, $F^{2n+1}$ in the last is the vertex superalgebra of $2n+1$ free fermions, and $h^\vee_\mathrm{1}$ (resp.\ $h^\vee_2$) denotes the dual Coxeter number of the Lie (super)algebra appearing in the left-hand side (resp.\ the right-hand side), namely $\g$, $\symp_{2n}$, $\osp_{1|2n}$ (resp.\ $\g$, $\osp_{1|2n}$, $\so_{2n+1}$). {Note that the ``missing" $\W_{B^+}^k(n,0)$ actually appears in the second isomorphism by Feigin-Frenkel duality. }
 
The triality conjecture was recently proven by the first two authors in \cite{CL1, CL2} up to taking $\Z_2$-orbifolds by outer automorphisms depending on the pairs $(X, Y)$. 
{As a next step, one might ask what the generalized Feigin-Frenkel duality \eqref{gen FF duality} can tell for the representation theory of the $\W$-superalgebras themselves.
It is already a non-trivial question beyond the original case as the duality \eqref{gen FF duality} is not for the $\W$-superalgebras themselves. Indeed, one is asked to enhance the duality \eqref{gen FF duality} to some reconstruction of the $\W$-superalgebras from one to the other. 
Such an example is established by two of the authors with N. Genra \cite{CGN}
for the pairs $\W^k_{X^+}(n,1)$ and $\W^\ell_{X^-}(n,1)$ ($X=A, D$), which led to a block-wise equivalence of their representation categories \cite{CGNS, FN23, FS}.
In this case, the reconstruction is realized as the coset construction  
\begin{align}\label{KS coset}
\W^k_{X^+}(n,1)\simeq \Com(\pi, \W^\ell_{X^-}(n,1)\otimes V_{\ssqrt{-1}\Z}),\quad \W^\ell_{X^-}(n,1)\simeq \Com(\pi, \W^k_{X^+}(n,1)\otimes V_{\Z})
\end{align}
by using the lattice vertex superalgebras $V_{\ssqrt{\pm1}\Z}$ and a Heisenberg vertex algebra $\pi$.}

In this paper, we establish the reconstruction theorem of the hook-type $\W$-superalgebras in the higher-rank cases as conjectured in \cite[Conjecture 1.2]{CL2}. 
In this generality, we need to reformulate \eqref{KS coset} since the naive coset construction does not work in the higher-rank case. 
A correct reformulation is the relative semi-infinite cohomology \cite{Fe, FGZ} against certain gluing objects.
(Interestingly, similar ideas appear in other situations; see e.g. \cite{Ara, CFL}.) 
{The basic idea here is to apply the $T$-duality \cite{CG, FG} for $\W^k_{X^+}(n,m)$ viewed as a vertex (super)algebra object in the Kazhdan--Lusztig category for the affine subalgebras $V^{k_\bg}(\bg)$ to obtain a new one, which will be identified with $\W^\ell_{Y^-}(n,m)$ sitting in the other side of the category by utilizing a certain uniqueness on the vertex algebraic structure.}
More concretely, we use, as gluing objects, vertex superalgebra extensions of two affine vertex algebras by Weyl modules of the form
\begin{align}\nonumber 
\KVA{N}[\g,k]=\bigoplus_{\lambda\in R}\weyl_\g^k(\lambda)\otimes \weyl_{\op{\g}}^{\op{k}}(\op{\lambda}^\dagger)\supset V^k(\g)\otimes V^{\op{k}}(\op{\g})
\end{align}
where levels $(k,\op{k})$ are related by the label $N\in \Z$; see Section \ref{sec Kernel VOA} for notation and details.
They are called the quantum geometric Langlands kernel vertex superalgebras in \cite[Section 7.3]{CG}, where their existence was conjectured. Later, they were constructed for irrational levels $k$ by Yuto Moriwaki \cite{M}. 
When $N=0$ and $\g=\op{\g}$, they are realized as the algebras of chiral differential operators over the corresponding reductive groups \cite{AG, GMS1, GMS2} and thus $\KVA{N}[\g,k]$ are regarded as ``twisted" versions of them.
Depending on $(X, Y)=(A, A), (C, C), (D, D)$, we consider the cases $\g= \gl_m, \so_{2m}, \mathfrak{sp}_{2m}$. 
For the remaining cases $(X,Y)=(B,O)$, $(O,B)$, we need a variant of $\KVA{N}[\g,k]$ extending $V^k(\so_{2m+1}) \otimes V^{\op{k}}(\osp_{1|2m})$, whose existence is known only for the case $m=1$ (see Section \ref{sec Kernel VOA}).
In this paper, we \emph{assume} the existence in these cases since the proof of the Main theorem below applies uniformly. 
To reconstruct the hook-type $\W$-superalgebras, we use the relative semi-infinite cohomology \cite{Fe, FGZ}, which satisfies the property
\begin{align*}
    \relsemicoh{m}\left( \g, \weyl_{\g}^k(\lambda)\otimes \weyl_{\g}^{\op{k}}(\mu^\dagger)\right)\simeq \delta_{m,0}\delta_{\lambda,\mu}
\end{align*}
for levels $(k,\op{k})$ satisfying $k+\op{k}=-2h^\vee$, see also Section \ref{Relative semi-infinite cohomology}.

Now, suppose $V$ is a vertex (super)algebra extension of $V^{k}(\g)$ at level $k$ lying in the Kazhdan--Lusztig category. 
By using the gluing object $\KVA{N}[\g,\op{k}]$, one obtains a new vertex superalgebra
\begin{align*}
   \check{V}= \relsemicoh{m}\left( \g, V\otimes \KVA{N}[\g,k]\right),
\end{align*}
which we call the convolution operation for $V$ with respect to $\KVA{N}[\g,k]$.
{The new superalgebra $\check{V}$ is slightly different from $V$.
Indeed, the affine subalgebra $V^{k}(\g)$ is replaced with $V^{\Breve{k}}(\op{\g})$, which sometimes changes the super-structure, that is, $\check{V}$ might be a vertex superalgebra even if $V$ is a vertex algebra. 
However, it is natural to expect that $V$ and $\check{V}$ share a similar algebraic structure as well as their representation categories.
The main theorem of the paper asserts that the pairs of hook-type $\W$-superalgebras satisfying the duality \eqref{gen FF duality} are indeed in such a situation.}

\begin{maintheorem} \textup{(Theorem \ref{Main results})}
For $k,\ell \in \C\backslash \Q$ under $r_X(k+h^\vee_{X^+})(\ell+h^\vee_{Y^-})=1$, there exist isomorphisms of vertex superalgebras
\begin{align}
\nonumber  &\W^\ell_{Y^-}(n,m)\simeq \relsemicoh{0}\left(\bg, \W^k_{X^+}(n,m)\otimes \KVA{1}[\bg,\alpha_\tp]\right),\\
 \nonumber &\W^k_{X^+}(n,m)\simeq \relsemicoh{0}\left(\bg, \W^\ell_{Y^-}(n,m)\otimes \KVA{-1}[\bg,\alpha_\tn]\right),
\end{align}
where $(\alpha_\tp,\alpha_\tn)$ are given by \eqref{alpha-index}. 
\end{maintheorem}
\noindent
{As an immediate application, we recover the duality \eqref{gen FF duality} by taking the affine cosets on both sides, which improves the previous restriction of taking $\Z_2$-orbifolds (Corollary \ref{duality theorem without orbifold}).}

One of the advantages of the convolution operation is the existence of the following obvious functor 
\begin{align*}
    V \text{-mod} \times \KVA{N}[\g,k]\text{-mod} \rightarrow \check{V} \text{-mod},\quad (M,N) \mapsto  \relsemicoh{0}(\g , M \otimes N).
\end{align*}
on the level of representation category. One might expect this functor to induce an equivalence of categories when restricted to appropriate full subcategories.
{For the $\W$-superalgebras appearing in the Main Theorem, the most fundamental functors are
\begin{align*}
    \mathbb{H}^\tp(M):= \relsemicoh{0}\left(\bg, M\otimes \KVA{1}[\bg,\alpha_\tp]\right),\quad \mathbb{H}^\tn(N):= \relsemicoh{0}\left(\bg, N \otimes \KVA{-1}[\bg,\alpha_\tn]\right).
\end{align*}
Then, one expects that these functors induce an equivalence of categories 
\begin{align}\label{equivalence}
   \mathbb{H}^\tp \colon \KL\left(\W^k_{X^+}(n,m)\right)  \simeq \KL\left(\W^\ell_{Y^-}(n,m)\right) \colon \mathbb{H}^\tn
\end{align}
between the category of $\W^k_{X^+}(n,m)$-modules which are in the Kazhdan-Lusztig category as $V^{k_\bg}(\bg)$-modules and the category of $\W^\ell_{Y^-}(n,m)$-modules similarly defined. 
Indeed, the block-wise equivalence between the representation categories for $\W^k_{X^+}(n,1)$ and $\W^\ell_{Y^-}(n,1)$ in \eqref{KS coset} was obtained in this way (``the block-wise" means that we may relax the conditions in \eqref{equivalence} to obtain more equivalences) \cite{CGNS, FN23}.
Moreover, one may compare further structures on the categories, such as the braided tensor structure, by establishing functorial isomorphisms on the space of intertwining operators.
This implies that the above equivalence \eqref{equivalence} should be a braided tensor equivalence if the intertwining operators induce the braided tensor structure in the sense of Huang--Lepowsky--Zhang (see \cite{M} and references therein). We refer \cite{C2} for a recent work in this direction.
After the first version of the present article was submitted, we were informed by G. Dhillon in a private communication that the equivalence \eqref{equivalence} can be interpreted as a vertex algebraic version of the mirabolic Satake correspondence \cite{BFGT, BFT1, BFT2}, which asserts an equivalence between the category of (certain) equivariant twisted $\mathscr{D}$-modules over the affine Grassmaian (appearing in the left-hand side) and the Kazhdan--Lusztig category of affine vertex superalgebra (appearing in the right-hand side).}
From the point of view of physics, it is natural to further improve the equivalence in terms of the conformal blocks or correlation functions in the spirit of the quantum geometric Langlands correspondence; see \cite{CHS, CH1, CH2, CH3} for such evidence based on the path integral approach for conformal field theory.
We aim to pursue this direction further and are currently studying the behavior of characters under our convolution operation.  

\noindent{\bf Acknowledgements} We thank Gurbir Dhillon, Naoki Genra and Yuto Moriwaki for very helpful discussions. We thank the anonymous referee for careful reading of the draft.
T.C. is supported by NSERC Grant Number RES0048511. A.L. is supported by Simons Foundation Grant 635650 and NSF Grant DMS-2001484. 
S.N. is supported by JSPS KAKENHI Grant Number 20J10147. 
This work was supported by World Premier International Research Center Initiative (WPI Initiative), MEXT, Japan.

\section{Kernel vertex superalgebras}
\subsection{Affine vertex superalgebras}\label{affine vertex algebra}
Let $\g=\g_{\bar{0}}\oplus \g_{\bar{1}}$ be a finite-dimensional Lie superalgebra equipped with an even supersymmetric invariant bilinear form $\kappa=(\cdot| \cdot)$. 
The affine vertex superalgebra $\widehat{\g}_\kappa$ is the central extension 
$$0\rightarrow \C K \rightarrow \widehat{\g}_\kappa \rightarrow \g[t^{\pm1}]\rightarrow 0 $$
of the loop superalgebra $\g[t^{\pm1}]=\g\otimes \C[t^{\pm1}]$ with Lie superbracket 
$$[X_n,Y_m]=[X,Y]_{n+m}+n(X|Y)\delta_{n+m,0}K$$
where we set $X_n:=Xt^n$ ($X \in \g, n \in \Z$).  
Consider the $\widehat{\g}$-module $V^\kappa(\g):=\mathrm{Ind}^{\widehat{\g}_\kappa}_{\g[t]\oplus \C K}\C$ induced from the one-dimensional representation $\C$, on which $\g[t]$ acts by zero and $K$ by identity.
It has a structure of vertex superalgebra, called the universal affine vertex superalgebra associated with $\g$ at level $\kappa$. As a vertex superalgebra, it is generated by the fields $Y(X_{-1}\mathbf{1},z):=X(z)=\sum_{n\in \Z}X_{n}z^{-n-1}$ ($X\in \g$) which satisfy the OPEs
$$X(z)Y(w)\sim \frac{[X,Y](w)}{(z-w)}+\frac{(X|Y)}{(z-w)^2}, \quad (X,Y\in \g),$$
or equivalently the $\lambda$-bracket \cite{DK}
\begin{align}\label{example of lambda bracket}
[X_\lambda Y]=[X,Y] +(X|Y)\lambda.
\end{align}
The vertex superalgebra $V^\kappa(\g)$ is $\Z_{\geq0}$-graded by setting $\Delta(X^1_{n_1}\cdots X^m_{n_m}\mathbf{1})=-\sum_{p=1}^m n_p$, ($X^i\in \g$). 
We denote by $\mathbf{KL}_\kappa(\g)$ the Kazhdan--Lusztig category of $V^\kappa(\g)$-modules, that is, the category consisting of all the ($\Z_2$-graded) positively-graded $V^\kappa(\g)$-modules which are integrable as $\g$-modules.

\subsection{Tensor representations}\label{tensor representation}
Here, we recollect and fix the notation of tensor representations of Lie superalgebras 
$\g=\gl_{m}, \sll_{m}, \so_{m}, \symp_{2m}, \osp_{1|2m}.$

We first consider the case when $\g$ is simple. Let $\g=\mathfrak{n}_+\oplus \mathfrak{h}\oplus \mathfrak{n}_-$ be the triangular decomposition, $\{\alpha_i\}_{i=1}^{\mathrm{rank}\g}\subset \mathfrak{h}^*$ the set of simple roots, and $\{\varpi_i\}_{i=1}^{\mathrm{rank}\g}$ the set of fundamental weights. 
We set $P=\oplus_i \Z\varpi_i$ (resp.\ $Q=\oplus_i \Z \alpha_i$) the weight lattice (resp.\ the root lattice). Then we denote by $P_+=\oplus_i\Z_{\geq0}\varpi_i$ the set of dominant integral weights and by $L_\lambda$, ($\lambda\in P_+$), the simple $\g$-module generated by highest weight vector of \emph{even parity} and of highest weight $\lambda$. In particular, $L_{\varpi_1}$ corresponds to the natural representation $\rho_\g$, which is $\C^m$ (resp.\ $\C^m$, $\C^{2m}$, $\C^{1|2m}$) when $\g=\sll_m$ (resp.\ $\so_m$, $\symp_{2m}$, $\osp_{1|2m}$). For $\lambda\in P_+$, let $\lambda^\dagger$ denote the highest weight for the dual representation $\Hom_{\C}(L_\lambda,\C^{1|0})$. By using the Weyl group $W$ of $\g$ and its longest element $w_{0}$, we have $\lambda^\dagger=-w_0\lambda$.
\begin{lemma}
For $\lambda\in P_+$, $\g$-module isomorphisms $L_{\lambda^\dagger}\simeq \Hom_\C(L_\lambda,\C^{1|0})$ are even.
\end{lemma}
\begin{proof}
It suffices to prove the case $\g=\osp_{1|2m}$. Decompose $L_\lambda$ as a $\symp_{2m}$-module: $L_\lambda|_{\symp_{2m}}\simeq \bigoplus_\mu n_{\lambda,\mu}L_\mu(\symp_{2m})$. Since the Cartan subalgebras of $\osp_{1|2m}$ and $\symp_{2m}$ coincide, $n_{\lambda,\mu}$ is non-zero only if $\lambda$ dominates $\mu$. Hence, the lowest weight of $L_\lambda$ coincides with that of $L_\lambda(\symp_{2m})$.
It follows that the lowest weight subspace, say $W$, is even, as the highest weight subspace is even by definition. {Now, the highest weight subspace of $\Hom_\C(L_\lambda,\C^{1|0})$ as a $\g$-module is $\Hom_\C(W,\C^{1|0})$ and thus of even parity. This completes the proof.}
\end{proof}
\noindent
Tensor representations of $\g$ are, by definition, modules appearing as submodules in the tensor algebra $T^\bullet (L_{\varpi_1}\oplus L_{\varpi_{1}^\dagger})$. We denote by $R\subset P_+$ the set of their highest weights. The condition is given in Table \ref{tab: cong condition} below. 
Note that when $\g\neq \sll_m$:  $L_{\lambda^\dagger}\simeq L_{\lambda}$ holds for all $\lambda\in R$ since $\varpi_1^\dagger=\varpi_1$ and that we have the following one-to-one correspondence 
\begin{align}\label{BO identification}
    R(\osp_{1|2m})\xrightarrow{\simeq} R(\so_{2m+1}),\quad \lambda=\sum_{i=1}^m\lambda_i \varpi_i\mapsto \op{\lambda}:= \sum_{i=1}^{m-1}\lambda_i \varpi_i+2\lambda_m \varpi_m
\end{align}
of highest weights under which the characters are identified, and thus, the branchings of tensor products coincide \cite{C, RS}.
By abuse of notation, we also denote the inverse of \eqref{BO identification} by $\lambda \mapsto \op{\lambda}$.

Next, we consider the case $\g=\gl_m$. We decompose $\gl_m=\sll_m\oplus \C \mathrm{Id}$ as usual and set the set of fundamental weights $\{\varpi_i\}_{i=1}^{m}$ so that $\{\varpi_i\}_{i=1}^{m-1}$ corresponds to that of $\sll_m$ and $\varpi_m$ to the dual of $\mathrm{Id}$.
Then we set $P_+=\oplus_{i=1}^{m-1}\Z_{\geq0}\varpi_i\oplus \Z \varpi_m$ and use the same notation as before.  
In particular, the natural representation $\rho_{\gl_m}=\C^m$ and its dual $\rho_{\gl_m}^\dagger=\overline{\C}^m$ correspond to $L_{\varpi_1+\varpi_m}$ and  $L_{\varpi_{m-1}-\varpi_m}$ respectively.

In Table \ref{tab: cong condition} below, we describe the set $R$ in terms of the coefficients $\lambda_i$ of the weights $\lambda= \sum \lambda_i \varpi_i$, {which is well-known for Lie algebras $\g$, e.g.\ \cite{FH}.
The condition for $\g=\osp_{1|2m}$ is obtained by using \eqref{BO identification}}.

\renewcommand{\arraystretch}{1.2}
\begin{table}[h]
    \caption{Condition for $R$}
    \label{tab: cong condition}
    \centering
    \begin{tabular}{|cc||cc| }
\hline
$\g$ & Condition &\emph{Cont'd} &  \\ \hline
$\gl_m$ &  $\lambda_{m}\equiv\sum_{i=1}^{m-1}i\lambda_i$ $\mod m$ & $\symp_{2m}$ & $\varnothing$\\ 
$\sll_{m}$ &$\varnothing$ &$\so_{2m}$ & $\lambda_{m-1}+\lambda_m\equiv 0$ $\mod 2$\\
$\so_{2m+1}$ & $\lambda_m\equiv 0$ $\mod 2$& $\osp_{1|2m}$ & $\varnothing$ \\ \hline
    \end{tabular}
\end{table}
\renewcommand{\arraystretch}{1}
Below, we will consider the following pairs $(\g, \op{\g})$ of Lie superalgebras 
$$(\gl_m,\gl_m),\quad (\so_{2m+1},\osp_{1|2m}),\quad (\symp_{2m},\symp_{2m}),\quad (\so_{2m},\so_{2m}),\quad (\osp_{1|2m}, \so_{2m+1}).$$
In the below, we will use the notation $\lambda\leftrightarrow \op{\lambda}$ in \eqref{BO identification} for all pairs by setting $\op{\lambda}=\lambda$ for the remaining cases.
 
\subsection{Kernel vertex superalgebras}\label{sec Kernel VOA}
Let us consider $\g$ as in \S 2.2, set $\kappa=k \kappa_0$ in terms of the normalized invariant bilinear form $\kappa_0$ on $\g$, and use $k$ instead of $\kappa$, for convenience, to express the level, say, $V^k(\g)$ instead of $V^{\kappa}(\g)$.
From now on, we assume $k\in \C\backslash \Q$ so that the Kazhdan--Lusztig category $\mathbf{KL}_k(\g)$ is semisimple. The simple modules are provided by the Weyl modules 
$$\mathbb{V}^k_{\g}(\lambda)=\mathrm{Ind}_{\g[t]\oplus \C K}^{\widehat{\g}_k}L_\lambda,\quad (\lambda\in P_+),$$ 
which will be denoted by $\mathbb{V}^k_\lambda$ when $\g$ is clear.
{As mentioned above in the Introduction, a key role will be played by
the following ind-object in the Deligne product $\mathbf{KL}_k(\g)\boxtimes \mathbf{KL}_{\op{k}}(\op{\g})$
\begin{align}\label{fake CDO}
    \KVA{n}[\g,k]=\bigoplus_{\lambda\in R}\weyl_\g^k(\lambda)\otimes \weyl_{\op{\g}}^{\op{k}}(\op{\lambda}^\dagger),\qquad (n\in \Z),
\end{align}
in the proof of the Main theorem.
Here, the levels $(k,\op{k})$ are related by the formula  
\begin{align}\label{final relations}
    \frac{1}{a(k+h^\vee_\g)}+\frac{1}{b(\op{k}+ h^\vee_{\op{\g}})}=c n
\end{align}
where $h^\vee_\g$ (resp.\ $h^\vee_{\op{\g}})$ is the dual Coxeter number of $\g$ (resp.\ $\op{\g}$) and $(a,b,c)$ are as in Table \ref{tab: abc} below.}
\renewcommand{\arraystretch}{1.2}
\begin{table}[h]
    \caption{}
    \label{tab: abc}
    \centering
    \begin{tabular}{|cccc|}
\hline
$\g$ & $(a,b,c)$ & Parity  & $\Delta_K$ \\ \hline
$\gl_m$ & $(1,1,1)$ & {\small $\Pi^{n} (\C^m\otimes \overline{\C}^m)$, $\Pi^{n} (\overline{\C}^m\otimes \C^m)$} &$\frac{nm}{2}$ \\
$\so_{2m+1}$ &  $(1,2,1)$& {\small $\C^{2m+1}\otimes \C^{1|2m}$} &$nm$\\
$\symp_{2m}$ & $(1,1,2)$ & {\small $\Pi^{n}(\C^{2m}\otimes \C^{2m})$} & $n(m+\frac{1}{2})$\\ 
$\so_{2m}$ &  $(1,1,1)$ &{\small $\Pi^{n}(\C^{2m}\otimes \C^{2m})$} & $n(m-\frac{1}{2})$\\
$\osp_{1|2m}$& $(2,1,1)$ &{\small $\C^{1|2m} \otimes \C^{2m+1}$} &$nm$  \\ \hline
    \end{tabular}
\end{table}
\renewcommand{\arraystretch}{1}

It was conjectured in \cite{CG, FG} that $\KVA{n}[\g,k]$ has the structure of a simple vertex superalgebra and was called the geometric Langlands kernel (super)algebra. 
The super-structure depends on the value $n$ so that the subspace
$$ \rho_\g\otimes \op{\rho}_{\op{\g}}\subset \weyl_\g^k(\rho_\g)\otimes \weyl_{\op{\g}}^{\op{k}}(\op{\rho}_{\op{\g}})
$$ (and it dual for $\g=\gl_m$) has the parity as described in Table \ref{tab: abc} where $\Pi$ denotes the odd isomorphism for arbitrary vector superspaces which interchanges the parity. The value $\Delta_K$ in Table \ref{tab: abc} is the conformal weight of the subspace $\rho_\g\otimes \op{\rho}_{\op{\g}}$ in terms of the conformal vector by Segal--Sugawara construction for $V^k(\g)\otimes V^{\op{k}}(\op{\g})$.

The construction of $\KVA{n}[\g,k]$ as vertex superalgebras was established in the case $\g=\op{\g}$ by Moriwaki \cite{M}, see also \cite[Corollary 1.4]{CKM} for some special cases. 
The construction uses the vertex algebras known as chiral differential operators \cite{AG, GMS1, GMS2}
$$\cdo:=\mathrm{Ind}_{\g[t]\oplus \C K}^{\widehat{\g}_k}\C[J_\infty G]$$
for the algebraic groups $G=SL_m, SO_{m}, Sp_{2m}$ whose Lie algebras we denote by $\g=\mathrm{Lie}(G)$ by abuse of notation. Here, $J_\infty G$ is the infinite jet scheme of $G$ on which  $\g[t]$ acts by the left-invariant vector fields. 
{As a vertex algebra, $\cdo$ is generated by the affine subalgebra $V^k(\g)$ and the commutative vertex algebra $\C[J_\infty G]$, which satisfy the OPEs
\begin{align*}
    Y(X,z) Y(f,w)\sim \frac{Y(X.f,w)}{(z-w)},\qquad (X\in \g, f \in \C[G]),
\end{align*}
where $X.f$ denotes the $\g$-action on $\C[G]$ as the left-invariant vector fields.}
Recall that $\C[G]$ is a $(\g,\g)$-bimodule by the left and right invariant vector fields, which provides the decomposition  
$\C[G]\simeq \bigoplus_{\lambda\in R}L_\lambda\otimes L_{\lambda^\dagger}$
known as the Peter--Weyl theorem. 
Similarly, $\cdo$ contains the affine vertex algebra $V^{\op{k}}(\g)$ chiralizing the right-invariant vector fields whose level $\op{k}$ is determined by the relation 
\begin{align}\label{duality rel for kernel VOA}
\frac{1}{(k+h^\vee_\g)}+\frac{1}{(\op{k}+h^\vee_\g)}=0.
\end{align}
As a $(V^k(\g), V^{\op{k}}(\g))$-bimodule, $\cdo$ decomposes into the right-hand side of \eqref{fake CDO} by setting $\op{\g}=\g$ and $\op{\lambda}=\lambda$. Note that the relation \eqref{duality rel for kernel VOA} is also equivalent to \eqref{final relations} for $\g=\symp_{2m}, \so_{2m}$ by setting $n=0$.
Therefore, $\KVA{n}[\g,k]$ can be regarded essentially as twisted versions of the chiral differential operators $\cdo$.

The coordinate ring $\C[G]$ has a $q$-deformation $O_q(G)$, called the quantum coordinate ring. It is a Hopf algebra dual to the quantum enveloping algebra $U_q(\g)$ and is a commutative algebra (ind-)object of the Deligne product $(U_q(\g), R_{\rho})\text{-mod}\boxtimes (U_{q^{-1}}(\g), R_{-\rho})\text{-mod}$. Here $(U_q(\g), R_{\rho})\text{-mod}$ is the braided tensor category of finite dimensional $U_q(\g)$-modules of type 1 equipped with a braiding defined by the R-matrix $R_{\rho}$ depending on $\rho\in \C$ such that $\mathrm{exp}(\pi \ssqrt{-1}\rho)=q$, see e.g. \cite{KS}.
In \cite{M}, by using $O_q(G)$ and twists of the braided tensor category $(U_q(\g), R_{\rho})\text{-mod}$ by abelian 3-cocycles, Moriwaki has constructed various kinds of simple vertex superalgebras which are extensions of affine vertex algebras and lattice vertex superalgebras including $\KVA{n}[\g,k]$ in \eqref{fake CDO} for $\op{\g}=\g$ via the Kazhdan--Lusztig correspondence 
$$\mathbf{KL}_{k}(\g) \xrightarrow{\simeq}\left(U_q(\g),R_{\rho}\right)\text{-mod}$$
where $\rho=1/r^\vee(k+h^\vee_\g)$ and $q=\mathrm{exp}(\pi \ssqrt{-1}/r^\vee(k+h^\vee_\g))$ with $r^\vee$ the lacing number of $\g$.
These twistings modify the chiral differential operators $\cdo$ so that the relation \eqref{duality rel for kernel VOA} is replaced by new ones
\begin{align}\label{new duality rel for kernel VOA}
\frac{1}{(k+h^\vee_\g)}+\frac{1}{(\op{k}+h^\vee_\g)}=r^\vee \tau_{R}\ n
\end{align}
where $\tau_{R}$ is the level of the lattice $R$, which is $m$ (resp.\ 1) if $\g=\sll_m$ (resp.\ $\so_m$, $\symp_{2m}$).
When $\g=\sll_m$, the twistings allow constructions of more vertex superalgebras by adding appropriate lattice vertex superalgebras. 
Let $V_{\sqrt{nm}\Z}$ denote the lattice vertex superalgebra associated with the lattice $\sqrt{nm}\Z$, which is an extension of the rank one Heisenberg vertex algebra $\pi$. This superalgebra has simple modules $V_{\frac{a}{\sqrt{nm}}+\sqrt{nm\Z}}$ $(a\in \Z_{nm})$.
Then, there exists a simple vertex superalgebra structure on 
\begin{align*}
\mathscr{V}^{n}:=\bigoplus_{\lambda\in R} \weyl_{\sll_m}^k(\lambda)\otimes \weyl_{\sll_m}^{\op{k}}(\lambda^\dagger)\otimes V_{\frac{n [\lambda]}{\sqrt{nm}}+\sqrt{nm}\Z}
\end{align*}
by \cite{M} where the level $\op{k}$ is determined by \eqref{new duality rel for kernel VOA} by setting $r^\vee=1$ and $\tau_{R}=1$. 
Here, $[\lambda]$ is the image of $\lambda$ under the projection $P\twoheadrightarrow P/Q\xrightarrow{\simeq} \Z_m, (\varpi_i\mapsto i)$. 
Note that the decomposition for $\mathscr{V}^n$ in the right-hand side is compatible with the Peter--Weyl theorem for $G=GL_m$ since weights appearing in the lattice part are regarded as one-dimensional representations for $\C\mathrm{Id}\subset \gl_m$.
It is straightforward to find that $\mathscr{V}^{n}\otimes \pi$ decomposes into the right-hand side of $\KVA{n}[\g,k]$ in \eqref{fake CDO}.
Now, to summarize, \cite[Main Theorem B]{M} specializes to the following statement:
\begin{theorem}[\cite{M}]\label{Kernel VOA of the first kind}
For $k\in \C\backslash \Q$, $\KVA{n}[\g,k]$ with $\g=\gl_m, \so_m$, or $\symp_{2m}$ has the structure of a simple vertex superalgebra satisfying the conditions in Table \ref{tab: abc}.
\end{theorem}
In \S \ref{Main section}, we will use the simple vertex superalgebras $\KVA{\pm1}[\g,k]$ for $\g=\gl_m$, $\so_{2m}$, $\symp_{2m}$ and also $\so_{2m+1}$, $\osp_{1|2m}$. 
However, the construction for the latter case is not known except for $m=1$.
The construction for $m=1$ is reviewed in the Remark \ref{lower rank kernels}.
We will treat the cases $\g=\so_{2m+1},\osp_{1|2m}$ under the assumption of their existence.

\begin{remark}\label{lower rank kernels}
\textup{The vertex superalgebras $\KVA{1}[\g,k]$ can be realized explicitly in lower rank cases by using the exceptional Lie superalgebra $\mathfrak{s}=\mathfrak{d}(2, 1; \alpha)$ ($\alpha \in \C$). Let $L_k(\mathfrak{s})$ denote the simple quotient of the associated affine vertex superalgebra $V^k(\mathfrak{s})$ and $\W_{k}(\mathfrak{s}, f_{\text{min}})$ the simple minimal $\W$-superalgebra at level $k$ \cite{KRW}. Then, we have the following realizations:
\begin{itemize}
\item $\KVA{1}[\gl_2,k]= L_1(\fd(2,1;-\frac{1}{k+1}))\otimes \pi$, see \cite[Theorem 9.1]{CG};
\item $\KVA{1}[\so_3,k]= L_1(\fd(2,1;-2k))$, see \cite[Equation (9.6)]{CG}, \cite[Theorem 3.2]{CGL2};
\item $\KVA{1}[\symp_2,k]= \W_{\frac{1}{2}}(\fd(2,1;-2k-3), f_{\text{min}})$, see \cite[Corollary 2.6]{CGL};
\item $\KVA{1}[\so_4,k]= \Hom_{\mathcal C({A_1^2})} (L_1(\fd(2,1;-\tfrac{1}{k+1}))^{\otimes 2}, V_{\mathbb Z^2})$, see \cite[Theorem 9.1]{CG};
\item $\KVA{1}[\osp_{1|2},k]=L_1(\fd(2,1;-\tfrac{1}{k+1}))$, see \cite[Equation (9.6)]{CG}, \cite[Theorem 3.2]{CGL2}.
\end{itemize}
\noindent 
In the above, $\mathcal C(L)$ denotes the category of modules of the lattice vertex algebra $V_L$ of the lattice $L$. {Then, $\KVA{-1}[\so_3,k](=\KVA{-1}[\osp_{1|2},\op{k}])$ is realized through the relative semi-infinite cohomology in \S \ref{Relative semi-infinite cohomology} as
$$\KVA{-1}[\so_3,k]\simeq \relsemicoh{0}\left(\g,\KVA{-2}[\so_3,k]\otimes \KVA{1}[\so_3,\ell]\right)$$
where $\ell$ is determined by $\op{k}+\ell=-2h^\vee_{\so_3}$.
(This construction also works for the other cases.)}
} 
\end{remark}

\begin{lemma}\label{tentative}
For $k\in \C\backslash \Q$, the bilinear form 
$$\langle\cdot|\cdot \rangle\colon  \Big(\rho_\g\otimes \rho_{\op{\g}}^\dagger\Big)\otimes \Big(\rho_\g^\dagger\otimes \rho_{\op{\g}}\Big)\rightarrow \C,\quad (x,y)\mapsto x_{(2\Delta_K-1)}y$$
is non-degenerate.
\end{lemma}
\proof
{It suffices to show that $\langle\cdot|\cdot \rangle$ is non-zero since $\langle\cdot|\cdot \rangle$ is a $(\g,\op{\g})$-invariant bilinear form.
Set $\A_\lambda=\weyl_\g^k(\lambda)\otimes \weyl_{\op{\g}}^{\op{k}}(\op{\lambda}^\dagger)$ so that $\A[\g,k]=\bigoplus_{R}\A_\lambda$. Let 
\begin{align*}
    Y_{\lambda, \mu}^\nu(\cdot, z)\colon \A_\lambda\otimes \A_\mu\rightarrow \A_\nu (\!(z)\!) 
\end{align*}
denote the intertwining operator among modules over $\A_0=V^k(\g)\otimes V^{\op{k}}(\op{\g})$ obtained as the restriction of the structure map $Y(\cdot,z)\colon \A[\g,k]\otimes \A[\g,k]\rightarrow \A[\g,k](\!(z)\!)$. Since $k\in \C\backslash \Q$, the space of all the intertwining operators is naturally identified as 
\begin{align}\label{intertwining op idneitification}
    I_{\A_0}\binom{\A_\nu}{\A_\lambda\ \A_\mu} \simeq \Hom_{(\g,\op{\g})}\big(L_{(\lambda, \op{\lambda})} \otimes L_{(\mu, \op{\mu})}, L_{(\nu, \op{\nu})} \big)
\end{align}
by \cite{FZ, Li} where we set $L_{(\lambda, \op{\lambda})}=L_\lambda\otimes L_{\op{\lambda}^\dagger}$ and so on. Recall that the isomorphism \eqref{intertwining op idneitification} is obtained by restricting the intertwining operators on the top spaces 
\begin{align*}
    \begin{array}{ccc}
      \A_\lambda\otimes \A_\mu &\rightarrow &\A_\nu \{\!\{z\}\!\}\\
      \cup && \cup\\
      L_{(\lambda, \op{\lambda})}\otimes L_{(\mu, \op{\mu})} & \rightarrow &L_{(\nu, \op{\nu})}\otimes \C z^\alpha
    \end{array}
\end{align*}
where $\A_\nu \{\!\{z\}\!\}$ denotes the $\A_\nu$-valued $z$-series with arbitrary complex powers $z^c$ ($c\in \C$) and $\alpha$ is a certain complex number.}

{Now, let $\lambda_o$ denote the highest weight corresponding  to $\rho_\g\otimes \rho_{\op{\g}}^\dagger$. Then $\langle\cdot|\cdot \rangle$ is the homomorphism in the right-hand side of \eqref{intertwining op idneitification} for $\lambda=\lambda_o$, $\mu=\lambda_o^\dagger$ and $\nu=0$. 
Hence, it suffices to show $Y_{\lambda_o, \lambda_o^\dagger}^0\neq 0$ to obtain $\langle\cdot|\cdot \rangle\neq0$. 
Since $\A[\g,k]$ is simple, the ideal $\mathscr{I}$ generated by $\A_{\lambda_o}$ is $\A[\g,k]$. 
On the other hand, the ideal $\mathscr{I}$ itself is described as 
\begin{align*}
    \mathscr{I}=\Span\{\A_{\mu}{}_{(N)}\A_{\lambda_o}\mid \mu \in P_+,\ N\in\Z\}
\end{align*}
by \cite[Proposition 4.5.6]{LL}. Since 
\begin{align*}
   \A_0=\A_0 \cap \mathscr{I}=\Span\{\A_{\lambda_o^\dagger}{}_{(N)}\A_{\lambda_o}\mid N\in\Z\},
\end{align*}
we have $Y_{\lambda_o, \lambda_o^\dagger}^0\neq 0$. This completes the proof.}
\endproof

\section{Relative semi-infinite cohomology}

\subsection{Relative semi-infinite cohomology}\label{Relative semi-infinite cohomology}
Let $\g$ be as in \S \ref{tensor representation} and take a parity-homogeneous basis $\{x_i\}_{i=1}^{\mathrm{dim}\ \g}$ with structure constants $c_{i,j}^k$, i.e. $[x_i,x_j]=\sum_kc_{i,j}^kx_k$. 

Introduce a vertex superalgebra $\semi$ generated by the fields 
$$\varphi_i^*(z)=\sum_{n\in \Z}\varphi_{i,n}^*z^{-n},\quad \varphi_i(z)=\sum_{n\in \Z}\varphi_{i,n}z^{-n-1},\quad (i=1,\dots, \mathrm{dim}\ \g)$$
which have the parity opposite to $x_i$ and satisfy the OPEs
$$\varphi_i(z)\varphi_j^*(w)\sim \frac{\delta_{i,j}}{z-w},\quad \varphi_i(z)\varphi_j(w)\sim 0\sim \varphi_i^*(z)\varphi_j^*(w).$$
{We introduce two gradings on $\semi$ as follows:
\begin{itemize}
    \item (Cohomological grading) $\semi=\bigoplus_{n\in \Z}\semiinf{n}{}$ so that 
    $$\deg\partial^n\varphi_i^*=1,\quad \deg\partial^n\varphi_i=-1,\quad \deg(ab)=\deg(a)+\deg(b),$$
    \item (Conformal grading) $\semi=\bigoplus_{\Delta\geq0}\semiinf{}{\Delta}$ so that 
    $$\Delta(\pd^n\varphi^*_i)=n,\quad \Delta(\pd^n\varphi_i)=n+1,\quad \Delta(ab)=\Delta(a)+\Delta(b).$$
\end{itemize}
Note that these gradings are compatible
\begin{align*}
    \semi=\bigoplus_{n,\Delta} \semiinf{n}{\Delta},\quad \semiinf{n}{\Delta}=\semiinf{n}{}\cap \semiinf{\bullet}{\Delta}
\end{align*}
so that $\semiinf{n}{\Delta}$ is finite-dimensional.}

{On the other hand, we have a homomorphism of vertex superalgebras 
\begin{align*}
    V^{\kappa_\g}(\g)\rightarrow \semi,\quad x_i(z)\mapsto \sum_{j,k=1}^{\dim{\g}}(-1)^{\bar{x}_j}c_{i,j}^k\NO{\varphi_k(z)\varphi_j^*(z)}
\end{align*}
where $\kappa_\g$ is the Killing form on $\g$.}
It preserves the cohomological grading by setting the trivial grading on $V^{\kappa_\g}(\g)$ and also the conformal grading by setting the standard one (i.e., $\Delta(x_i)=1$) on $V^{\kappa_\g}(\g)$. 
In particular, $\semi$ is a positive energy representation over $V^{\kappa_\g}(\g)$, and moreover, is an object in $\mathbf{KL}_{\kappa_{\g}}(\g)$ as each subspace $\semiinf{n}{\Delta}$ is a finite-dimensional $\g$-module.

For an object $M$ in $\KL_{-\kappa_\g}$, we introduce 
$$C^{\frac{\infty}{2}+\bullet}(\g,M):=M\otimes \semi,$$
which is a $V^0(\g)$-modules with respect to the diagonal action.
It is a complex by setting an (odd) differential 
$$d=\int Q(z)dz$$
with 
$$Q(z)= \sum_i (-1)^{\bar{x}_i}\NO{x_i(z)\varphi_i^*(z)}-\frac{1}{2}\sum_{i,j,k} (-1)^{\bar{x}_i \bar{x}_k}c_{i,j}^k\NO{\varphi_i^*(z)\varphi_j^*(z)\varphi_k(z)}$$
and a grading given by that of $\semi$.
The cohomology $\semicoh{\bullet}(\g,M)$ is called the semi-infinite cohomology of $\widehat{\g}$ with coefficients in $M$ \cite{Fe, FGZ}. 
In the below, we will consider a relative version, called \emph{the relative semi-infinite cohomology},  which assigns the cohomology of the following subcomplex: 
\begin{align}\label{rel semi complex}
C^{\frac{\infty}{2}+\bullet}_{\operatorname{rel}}(\g,M):=\left(M\otimes \relsemi(\g) \right)^\g.
\end{align}
Here $\relsemi(\g)\subset \semi$ is the subalgebra generated by $\{\partial \varphi_i^*(z), \varphi_i(z)\}_{i=1}^{\mathrm{dim}\ \g}$ and $(?)^\g$ denotes the $\g$-invariant subspace by the diagonal $\g$-action.
By associating the relative semi-infinite cohomology $\rel{\bullet}(\g,M)$ for each $M$, we obtain a functor 
$$\rel{\bullet}(\g,\cdot)\colon \KL_{-\kappa_\g} \rightarrow \mathrm{SVect}^\Z,$$
where $\mathrm{SVect}^\Z$ denotes the category of $(\Z_2\times\Z)$-graded vector spaces over $\C$.

\begin{remark}\label{decomposition of complexes}
\textup{{
Since $d$ is of conformal degree 0, the decompositions 
$$C^{\frac{\infty}{2}+\bullet}(\g,M)=\bigoplus C^{\frac{\infty}{2}+\bullet}(\g,M)_\Delta\supset C^{\frac{\infty}{2}+\bullet}_{\operatorname{rel}}(\g,M)=\bigoplus C^{\frac{\infty}{2}+\bullet}_{\operatorname{rel}}(\g,M)_\Delta$$
by the conformal grading are decompositions into subcomplexes. Moreover, the subcomplexes consist of finite dimensional terms 
\begin{align*}
    \dim C^{\frac{\infty}{2}+n}(\g,M)_\Delta,\quad \dim C^{\frac{\infty}{2}+n}_{\operatorname{rel}}(\g,M)_{\Delta}<\infty. 
\end{align*}
For the latter, we have $\dim C^{\frac{\infty}{2}+\bullet}_{\operatorname{rel}}(\g,M)_\Delta< \infty$ more strongly.}}
\end{remark}

For a finite dimensional $\g$-module $M$, we have an even natural isomorphism
\begin{align*}
    \Hom_\C(M,\C^{1|0})\otimes M\xrightarrow{\simeq}\End(M),\quad f\otimes v\mapsto [w\mapsto f(w)v].
\end{align*}
Let $\mathrm{str}_{M}$ denote the pull-back of $\id_M$ under this isomorphism. When $M=L_\lambda$, we also write $\str_M=\str_\lambda$.

\begin{lemma}\label{multi of trivial rep}
For $\lambda,\mu \in P_+$, we have an even isomorphism of vector superspaces
$$\Hom_\g(\C^{1|0}, L_\lambda\otimes L_\mu)\simeq \delta_{\lambda,\mu^\dagger}\C\mathrm{str}_\mu.$$
\end{lemma}
\proof
It is immediate from Schur's lemma: 
$$\notag\Hom_\g(\C^{1|0}, L_\lambda\otimes L_\mu)\simeq \Hom_\g(L_{\lambda^\dagger}, L_\mu)\simeq \delta_{\lambda^\dagger,\mu}\C\mathrm{id}_{L_\mu}.$$
\endproof
The relative semi-infinite cohomology satisfies the following property.
\begin{theorem}[\cite{CFL,FGZ}]\label{main theorem of relcoh}
For $\lambda,\mu \in P_+$, 
$$\relsemicoh{n}\left(\g,\mathbb{V}^{\kappa_1}_\lambda\otimes \mathbb{V}^{\kappa_2}_\mu\right)\simeq \delta_{n,0}\delta_{\lambda,\mu^\dagger}\C [\mathrm{str}_{\mu}]$$
holds for {the levels $(\kappa_1,\kappa_2)$ such that $\kappa_1+\kappa_2=-\kappa_\g$} and the restriction of $\kappa_i$ to the semisimple part of $\g_i$ is an irrational level and to the abelian part is nonzero.
\end{theorem} 
\noindent
This theorem is proven in \cite{CFL, FGZ} for simple Lie algebras and in 
 \cite[Proposition 6.1]{CGNS} for abelian Lie algebras. 
We give a detailed proof in general, including the super case $\g=\osp_{1|2m}$ in Appendix \ref{proof of formality of relcoh}, by following \cite{FGZ}.
\begin{remark}\textup{
Thanks to the restriction \eqref{rel semi complex}, Theorem \ref{main theorem of relcoh} holds also for Wakimoto modules $\mathbb{W}_\lambda^k$ \cite{FF2,Fr}:
$$\relsemicoh{n}\left(\g,\mathbb{W}^k_\lambda\otimes \mathbb{W}^\ell_\mu\right)\simeq \relsemicoh{n}\left(\g,\mathbb{W}^k_\lambda\otimes \mathbb{V}^\ell_\mu\right) \simeq \delta_{n,0}\delta_{\lambda,\mu^\dagger}\C [\mathrm{str}_{\mu}]$$
{for Lie algebras $\g$ and levels $k, \ell\in \C\backslash \Q$ satisfying $k+\ell=-2h^\vee$}
since $\mathbb{W}_\lambda^k \simeq \mathrm{Ind}_{\g[t]\oplus \C K}^{\widehat{\g}_k}M^\vee_\lambda$ for $k\in \C\backslash \Q$. Here, $M^\vee_\lambda$ is the contragredient Verma module over $\g$ with highest weight $\lambda$.}
\end{remark}
\subsection{Bilinear form}
Here, we prove an easy but important fact on a pairing on vertex superalgebras.

Let $V$ be an arbitrary vertex superalgebra and suppose that we have parity-homogeneous elements $a_{i}$, $b_i$ ($i=1,2$) of $V$ satisfying the OPEs
\begin{align}\label{assumption of the lambda}
  { Y(a_1,z)Y(a_2,w)\sim \frac{\alpha}{(z-w)^{N+1}}+\cdots,\quad Y(b_1,z)Y(b_2,w)\sim \frac{\beta}{(z-w)^{M+1}}+\cdots,\quad Y(a_i,z)Y(b_j,w)\sim 0}
\end{align}
for some $\alpha,\beta\in \C$ and $N,M\geq0$. Here, we mean by ``$\cdots$'' the terms of lower order poles with respect to $(z-w)$. 

\begin{lemma}\label{inner form}
Under the assumption \eqref{assumption of the lambda}, we have 
\begin{align*}
    Y(a_1b_1,z)Y(a_2b_2,w)\sim \frac{(-1)^{\bar{b}_1\bar{a}_2}\alpha\beta}{(z-w)^{N+M+2}}+\cdots.
\end{align*} 
\end{lemma}
\proof {
For the convenience, we show the formula by using the $\lambda$-brackets $[u_\lambda v]=\sum_{n\geq0}u_{(n)}v\lambda^{(n)}$ \cite{DK} instead of the OPEs where $\lambda^{(n)}=\lambda^n/n!$. Then, We have
\begin{align*}
    \lbr{a_1}{a_2}= \alpha \lm{N}+o(\lambda^N),\quad \lbr{b_1}{b_2}=\beta \lm{M} +o(\lambda^M),\quad \lbr{a_i}{b_j}=0,
\end{align*}
which are equivalent to \eqref{assumption of the lambda} and show
\begin{align}\label{formula we want}
\lbr{a_1b_1}{a_2b_2}=(-1)^{\bar{b}_1\bar{a}_2}\alpha\beta \lm{N+M+1}+o(\lambda^{N+M+1}).
\end{align}
Recall that the $\lambda$-brackets satisfy the following formulae \cite{DK}
\begin{align*}
    &\lbr{\pd u}{v}=-\lambda \lbr{u}{v},\\
    &\lbr{u}{vw}=\lbr{u}{v}w+(-1)^{\bar{u}\bar{v}}v\lbr{u}{w}+\int^\lambda_0[\lbr{u}{v}_\mu w]d \mu,\\
    &\lbr{uv}{w}=e^{\pd_u\pd_\lambda}u \lbr{v}{w}+(-1)^{\bar{u}\bar{v}} \left(e^{\pd_v\pd_\lambda}v\lbr{u}{w}+\int_0^\lambda[v_\mu[u_{\lambda-\mu}w]]d \mu\right),\\
    &\lbr{u}{[v{}_\mu w]}-(-1)^{\bar{u}\bar{v}}[v{}_\mu \lbr{u}{w}]=[\lbr{u}{v}{}_{\lambda+\mu}w],
\end{align*}
where $e^{x}=\sum_{n=0}^{\infty}x^n/n!$ and $\pd_u$ denotes the differential acting only on the term $u$.
Then, we have
\begin{align*}
\lbr{(a_1b_1)}{(a_2b_2)}
&=\lbr{(a_1b_1)}{a_2}b_2+(-1)^{(\bar{a}_1+\bar{b}_1)\bar{a}_2}a_2\lbr{(a_1b_1)}{b_2}+\int^\lambda_0[\lbr{(a_1b_1)}{a_2}{}_\mu b_2]d \mu.
\end{align*}
On the other hand, we have  
\begin{align*}
&\lbr{(a_1b_1)}{a_2}b_2=(-1)^{\bar{a}_1\bar{b}_1}\left(e^{\pd_{b_1}\pd_\lambda}b_1 \lbr{a_1}{a_2} \right)b_2=\oo{N+1},\\
&a_2\lbr{(a_1b_1)}{b_2}=(-1)^{(\bar{a}_1+\bar{b}_1)\bar{a}_2}a_2e^{\pd_{a_1}\pd_\lambda}a_1\lbr{b_1}{b_2}=\oo{M+1},
\end{align*}
and 
\begin{align*}
\int^\lambda_0[ \lbr{(a_1b_1)}{a_2}{}_\mu b_2]d \mu&=(-1)^{\bar{a}_1\bar{b}_1}\int_0^\lambda \left[e^{\pd_{b_1}\pd_\lambda}b_1 \lbr{a_1}{a_2}{}_\mu b_2\right] d \mu \\
&=(-1)^{\bar{a}_1\bar{b}_1}\int_0^\lambda \left[e^{\pd\pd_\lambda}b_1 e^{-\pd\pd_\lambda}\lbr{a_1}{a_2}{}_\mu b_2\right] d \mu \\
&=(-1)^{\bar{b}_1\bar{a}_2}\int_0^\lambda e^{-\pd\pd_\lambda}\left( \lbr{a_1}{a_2}\right)  e^{-\mu\pd_\lambda}[b_{1\mu} b_2] d \mu\\
&=(-1)^{\bar{b}_1\bar{a}_2}\int_0^\lambda e^{-\pd\pd_\lambda}\left(\alpha\lm{N}+\oo{N}\right) e^{-\mu\pd_\lambda}\left(\beta\mu^{(M)}+\oom{M}\right) d \mu\\
&=(-1)^{\bar{b}_1\bar{a}_2}\alpha\beta \int_0^\lambda (\lambda-\mu)^{(N)}\mu^{(M)}d \mu+\oo{N+M+1}\\
&=(-1)^{\bar{b}_1\bar{a}_2}\alpha\beta\lm{N+M+1}+\oo{N+M+1}.
\end{align*}
Hence, we obtain the formula \eqref{formula we want} as 
\begin{align*}
\lbr{(a_1b_1)}{(a_2b_2)}
&=\oo{N+1}+\oo{M+1}+(-1)^{\bar{b}_1\bar{a}_2}\alpha\beta\lm{N+M+1}+\oo{N+M+1}\\
&=(-1)^{\bar{b}_1\bar{a}_2}\alpha\beta\lm{N+M+1}+\oo{N+M+1}.
\end{align*}}
\endproof
Now, let $\g$ and $\kappa_i$ be as in Theorem \ref{main theorem of relcoh} and suppose that we have $\frac{1}{2}\Z$-graded vertex superalgebras
\begin{align*}
\mathscr{A}_i=\bigoplus_{\lambda\in R}\mathscr{C}_\lambda^i\otimes \weyl_\lambda^{\kappa_i}(\g),\quad (i=1,2)
\end{align*}
with $R\subset P_+$  
such that 
\begin{itemize}
    \item[(A1)] the affine cosets $\mathscr{C}_0^i\subset \mathscr{A}_i$ are $\frac{1}{2}\Z_{\geq0}$-graded vertex superalgebras of CFT type:
    $$\mathscr{C}_0^i=\bigoplus_{\Delta \in \frac{1}{2}\Z_{\geq0}} \mathscr{C}_{0,\Delta}^i,\qquad \mathscr{C}_{0,0}^i=\C \mathbf{1},$$
    \item[(A2)] $\mathscr{C}^i_{\lambda}$ are all positive energy representations of $\mathscr{C}_0^i$.
\end{itemize}
{By (A2), one may take the lowest conformal weight subspace $\mathsf{C}_{\lambda}^i\subset \mathscr{C}_\lambda^i$.
For each $\lambda$ and $i$, we have a unique integer, say, $K$ such that the structure map $\mathscr{A}_i\otimes \mathscr{A}_i\rightarrow \mathscr{A}_i(\!(z)\!)$ induces 
\begin{align*}
    \left({\sf C}_\lambda^i\otimes L_\lambda \right)\otimes\left({\sf C}_{\lambda^\dagger}^i\otimes L_{\lambda^\dagger}\right)\rightarrow \C \mathbf{1},\quad (u,v)\mapsto u_{(K)}v,
\end{align*}
thanks to (A1). By identifying $\C \mathbf{1}$ with $\C$, it defines a bilinear form 
\begin{align*}
\langle\cdot|\cdot \rangle_{\lambda}^i\colon \left({\sf C}_\lambda^i\otimes L_\lambda \right)\otimes\left({\sf C}_{\lambda^\dagger}^i\otimes L_{\lambda^\dagger}\right)\rightarrow \C,\quad  (i=1,2).
\end{align*}
Note that it is $\g$-invariant as 
\begin{align*}
    \langle xu| v\rangle_{\lambda}^i+ (-1)^{\bar{u}\bar{x}} \langle u| xv\rangle_{\lambda}^i
    &= (x_{(0)}u)_{(K)}v+(-1)^{\bar{u}\bar{x}} u_{(K)}(x_{(0)}v)\\
    &=x_{(0)}(u_{(K)}v)\\
    &=0
\end{align*}
holds for $x\in \g$. Here we have used $x_{(0)}\mathbf{1}=0$. }
Similarly, we consider the relative semi-infinite cohomology
$$\relsemicoh{0}\left(\g, \mathscr{A}_1\otimes \mathscr{A}_2\right)\simeq \bigoplus_{\lambda\in R}\mathscr{C}_\lambda^1\otimes \mathscr{C}_{\lambda^\dagger}^2,$$
and obtain bilinear forms
$$\langle\cdot|\cdot \rangle_{\lambda}\colon \left({\sf C}_{\lambda}^1\otimes {\sf C}^2_{\lambda^\dagger}\right)\otimes \left({\sf C}^1_{\lambda^\dagger}\otimes {\sf C}^2_{\lambda}\right)\rightarrow \C,\quad (\lambda \in R).$$

\begin{corollary}\label{nondegeneracy under relcoh}
If $\langle\cdot|\cdot \rangle_{\lambda}^i$ $(i=1,2)$ are non-degenerate, then so is $\langle\cdot|\cdot \rangle_{\lambda}$.
\end{corollary}
\proof
By  Lemma \ref{inner form}, we obtain $(\g,\g)$-invariant bilinear forms on the complex 
\begin{align}\label{tensored form}
\langle\cdot|\cdot \rangle^1_\lambda\otimes \langle\cdot|\cdot \rangle^2_{\lambda^\dagger}\colon \left( {\sf C}^1_{\lambda}\otimes  L_{\lambda}\otimes L_{\lambda^\dagger}\otimes {\sf C}^2_{\lambda^\dagger}\right) \otimes \left({\sf C}^1_{\lambda^\dagger}\otimes L_{\lambda^\dagger}\otimes L_\lambda\otimes {\sf C}^2_{\lambda}\right) \rightarrow \C,\qquad (\lambda \in R).
\end{align}
By Theorem \ref{main theorem of relcoh}, the bilinear form $\langle\cdot|\cdot \rangle_{\lambda}$ is the restriction of $\langle\cdot|\cdot \rangle^1_\lambda\otimes \langle\cdot|\cdot \rangle^2_{\lambda^\dagger}$ to the subspaces
\begin{align}\label{realization of coh class}
\left({\sf C}_{\lambda}^1\otimes [\str_\lambda]\otimes{\sf C}^2_{\lambda^\dagger}\right)\otimes \left({\sf C}^1_{\lambda^\dagger}\otimes [\str_{\lambda^\dagger}]\otimes {\sf C}^2_{\lambda}\right)\rightarrow \C.
\end{align}
Since the subspaces are exactly the $\g$-invariant subspaces under the diagonal action.
Since arbitrary $\g$-invariant bilinear forms $L_\mu\otimes L_\nu\rightarrow \C$ are zero unless $\mu=\nu^\dagger$, the non-degeneracy of \eqref{tensored form} implies that of $\langle\cdot|\cdot \rangle_{\lambda}$ on \eqref{realization of coh class}.
\endproof

\section{Duality in hook-type $\W$-superalgebras}\label{Main section}
\subsection{Hook-type $\W$-superalgebras}\label{Hook-type}
Let $\g$ be a simple Lie algebra or, more generally, a basic-classical simple Lie superalgebra, and take $\kappa=k \kappa_0$ be an even supersymmetric invariant bilinear form on $\g$. We denote by $V^\kappa(\g)$ the associated affine vertex superalgebra as in \S \ref{affine vertex algebra}. For a pair of an even nilpotent element $f$ and a good grading $\Gamma: \g = \bigoplus_{j \in \frac{1}{2}\Z} \g_j$, we associated the $\W$-superalgebra $\W^\kappa(\g, f; \Gamma)$ obtained as the quantum Drinfeld-Sokolov reduction of $V^\kappa(\g)$ with respect to $(\g,\Gamma)$ \cite{FF1, KRW}.
We choose the Dynkin grading as a good grading $\Gamma$ for $f$, that is, the eigenspace decomposition of the adjoint action of $x$ in an $\sll_2$-triple $(e,h=2x,f)$ containing $f$. Thus, the $\W$-superalgebra will depend only on the nilpotent element $f$. 

The hook-type $\W$-superalgebras are $\W$-superalgebras studied in \cite{CL1,CL2, GR} and will be denoted by 
$$\W^k_{A^\pm}(n,m),\quad \W^k_{B^\pm}(n,m),\quad \W^k_{C^\pm}(n,m),\quad \W^k_{D^\pm}(n,m),\quad \W^k_{O^\pm}(n,m),$$
as listed in Table \ref{tab: hopok-type Walg-intro} in this paper. 
Here we express the level as $\kappa=k \kappa_0$ where $\kappa_0$ is the bilinear form in Table \ref{tab: hopok-type Walg-intro} and \emph{assume $k$ to be irrational, i.e. $k\in \C\backslash \Q$} otherwise stated.  
These $\W$-superalgebras are parameterized by embeddings $\ag\oplus\bg\subset \g$ with nilpotent element taken to be principal inside $\ag$ so that 
$$\g\simeq
\begin{cases}
(\ag\oplus \bg)\oplus \rho_\ag\otimes \rho^\dagger_\bg\oplus \rho^\dagger_\ag\otimes \rho_\bg & \text{if}\ \g=\sll_{N|M},\\
(\ag\oplus \bg)\oplus \rho_\ag\otimes \rho_\bg& \text{if}\ \g=\osp_{N|2M},
\end{cases}$$
holds where $\rho_\ag$ and $\rho^\dagger_\ag$ (resp.\ $\rho_\bg$ and $\rho^\dagger_\bg$) are the natural representation of $\ag$ (resp.\ $\bg$) and its dual. 
\renewcommand{\arraystretch}{1.2}
\begin{table}[h]
    \caption{Hook-type $\W$-superalgebras}
    \label{tab: hopok-type Walg}
    \centering
\begin{adjustwidth}{0cm}{}
\begin{tabular}{|cccccc|}
\hline
Label& $\W^k_{A^+}(n,m)$ &$\W^k_{B^+}(n,m)$ &$\W^k_{C^+}(n,m)$ &$\W^k_{D^+}(n,m)$ &$\W^k_{O^+}(n,m)$\\ \hline
$\g$       & $\sll_{n+m}$ &$\so_{2(n+m+1)}$                 &$\symp_{2(n+m)}$& $\so_{2(n+m)+1}$& $\osp_{1|2(n+m)}$\\ 
$\kappa_0$ &$\tr$          &$\frac{1}{2}\tr$                  &$\tr$ & $\frac{1}{2}\tr$& $-\str$\\
$h^\vee$ &$n+m$        &$2(n+m)$                        &$n+m+1$& $2(n+m)-1$& $n+m+\frac{1}{2}$\\
$\ag$      &$\sll_n$      &$\so_{2n+1}$                         &$\symp_{2n}$& $\so_{2n+1}$& $\symp_{2n}$\\ 
$\bg$      &$\gl_m$      &$\so_{2m+1}$                          &$\symp_{2m}$& $\so_{2m}$& $\osp_{1|2m}$\\
$k_{\bg}^+$  &{ $k+n-1$}            &    { $k+2n$}                  &    { $k+n-\frac{1}{2}$}      & {$k+2n$}       &{$k+n-\frac{1}{2}$}\\ \hline \hline
Label& $\W^k_{A^-}(n,m)$ &$\W^k_{B^-}(n,m)$ &$\W^k_{C^-}(n,m)$ &$\W^k_{D^-}(n,m)$ &$\W^k_{O^-}(n,m)$\\ \hline
$\g$       & $\sll_{n+m|m}$ &$\osp_{2m+1|2(n+m)}$             &$\osp_{2(n+m)+1|2m}$     & $\osp_{2m|2(n+m)}$& $\osp_{2(n+m)+2|2m}$\\ 
$\kappa_0$ &$\str$            &$-\str$                                &$\frac{1}{2}\str$                   & $-\str$& $\frac{1}{2}\str$\\
$h^\vee$ &$n$                &$n+\frac{1}{2}$                     &$2n-1$                          & $n+1$& $2n$\\
$\ag$      &$\sll_{n+m}$     &$\symp_{2(n+m)}$                    &$\so_{2(n+m)+1}$          & $\symp_{2(n+m)}$& $\so_{2(n+m)+1}$\\ 
$\bg$      &$\gl_m$          &$\so_{2m+1}$                   &$\symp_{2m}$            & $\so_{2m}$& $\osp_{1|2m}$\\
$k_{\bg}^-$  &{ $-(k+n+m)+1$}  &{ $-2(k+n+m)+1$ }       &{ $-(\frac{1}{2}k+n+m)$}        &   {$-2(k+n+m)+1$}     &{$-(\frac{1}{2}k+n+m)$}\\ \hline
\end{tabular}
\end{adjustwidth}
\end{table}
\renewcommand{\arraystretch}{1}
The $\W$-superalgebra $\W^k_{X^\pm}(n,m)$ contains $V^{k_\bg^\pm}(\bg)$ as a (usually maximal) affine vertex subalgebra. {The dual Coxeter number of $\g$ is denoted by $h^\vee_{X^\pm}$.} 
Let us denote by $C_{X^\pm}^k(n,m)$ the coset subalgebra
$$C_{X^\pm}^k(n,m)=\Com\left(V^{k_\bg^\pm}(\bg),\W^k_{X^\pm}(n,m)\right).$$
For $X=B$ (resp.\ $D,O$), the $\W$-superalgebra $\W^k_{X^\pm}(n,m)$ has an automorphism group $O_{2m+1}$ (resp.\ $O_{2m}$, $OSp_{1|2m}$),
which is induced by their adjoint action on $\g$. They give rise to an automorphism group $\Z_2\simeq O_{2m+1}/SO_{2m+1}$ (resp.\ $O_{2m}/SO_{2m}$, $OSp_{1|2m}/SOSp_{1|2m}$) on $C^k_{X^\pm}(n,m)$. We set 
\begin{align*}
    D^k_{X^\pm}(n,m)=\begin{cases}
    C^k_{X^\pm}(n,m),& \text{if}\ X=A,C,\\
    C^k_{X^\pm}(n,m)^{\Z_2},& \text{if}\ X=B,D,O.
    \end{cases}
\end{align*}
Then $\W^k_{X^\pm}(n,m)$ satisfies the following properties:
\begin{itemize}
    \item[(P1)] It has a free strong generating set of type
    \begin{align*} 
        \begin{cases}\W\left(1^{\dim \bg},e_1,e_2,\cdots,e_{\mathrm{rank}\ag},\left(\frac{\dim \rho_\ag+1}{2}\right)^{2\dim \rho_\bg} \right)& \text{if}\ X=A,\\
        \W\left(1^{\dim \bg},e_1,e_2,\cdots,e_{\mathrm{rank} \ag}, \left(\frac{\dim \rho_\ag+1}{2}\right)^{\dim \rho_\bg}\right)& \text{if}\ X=B,C,D,O.
        \end{cases}
    \end{align*}
    The first $\dim \bg$ elements of conformal dimension $\Delta=1$ correspond to a basis of $\bg$ and generate $V^{k_\bg^\pm}(\bg)$. The second set which we denote by $S_{\sf coset}=\{W_{X^\pm}^i\mid i=e_1,e_2,\cdots,e_{\mathrm{rank} \ag}\}$ are all even and have conformal weights $\Delta=e_1,e_2,\cdots,e_{\mathrm{rank} \ag}$, which are exponents of $\ag$ {plus one}. 
They form a part of a strong generating set of $C^k_{X^\pm}(n,m)$ and $W_{X^\pm}^2$ is the conformal vector of $C^k_{X^\pm}(n,m)$. The third set which we denote by $S_{\sf nat}\sqcup S_{\sf nat}^\dagger$ (resp. $S_{\sf nat}$) if $X=A$ (resp.\  if $X=B,C,D,O$), corresponds to a standard basis of $\rho_\bg\oplus \rho_\bg^\dagger$ (resp.\ $\rho_\bg$). 
    They span a vector subspace $V_{\sf nat}\oplus V_{\sf nat}^\dagger=\C S_{\sf nat}\oplus \C S_{\sf nat}^\dagger$ (resp.\ $V_{\sf nat}=\C S_{\sf nat}$) whose parity and conformal weight $\Delta_\rho$ are as in Table \ref{tab: primary fields} and generate Weyl modules $\weyl^{k_\bg^\pm}_{\rho_\bg}(\bg) \oplus \weyl^{k_\bg^\pm}_{\rho_\bg^\dagger}(\bg)$ (resp.\ $\weyl^{k_\bg^\pm}_{\rho_\bg}(\bg)$).
    \item[(P2)] For $k\in \C \backslash \Q$, we may normalize the elements of $S_{\sf nat}\sqcup S_{\sf nat}^\dagger$ (resp. $S_{\sf nat}$) so that the bilinear form 
    $$V_{\sf nat}\otimes V_{\sf nat}\rightarrow \C,\quad (a,b)\mapsto a_{(2\Delta_\rho-1)}b$$
coincides with the natural $\bg$-invariant bilinear form
\begin{align*}
    \begin{cases}
    \langle\cdot|\cdot\rangle\colon (\rho_\bg\oplus \rho_\bg^\dagger)\otimes (\rho_\bg\oplus \rho_\bg^\dagger)\rightarrow \C& \text{if}\ X=A,\\
    \langle\cdot|\cdot\rangle\colon \rho_\bg \otimes \rho_\bg\rightarrow \C& \text{if}\ X=B,C,D,O,
    \end{cases} 
\end{align*}
{where we use $\rho_\bg \simeq \rho_\bg^\dagger$ for the second case.}
    Then the elements in $S_{\sf nat}\sqcup S_{\sf nat}^\dagger$ (resp. $S_{\sf nat}$) satisfy the OPEs
\begin{align*}
    A(z)B(w)\sim \frac{\langle A| B\rangle}{(z-w)^{2\Delta_\rho}}+\cdots. 
\end{align*}
\end{itemize}
\renewcommand{\arraystretch}{1.3}
\begin{table}[h]
    \caption{Primary fields}
    \label{tab: primary fields}
    \centering
\begin{tabular}{|cccccc|}
\hline
Label& $A^+$ & $B^+$ &$C^+$ &$D^+$ &$O^+$\\ \hline
Parity & $\C^{m}\oplus \overline{\C}^m$ & $\C^{2m+1}$ &$\C^{2m}$ & $\C^{2m}$ & $\Pi\C^{1|2m}$\\ 
$\Delta_\rho$ &$\frac{n+1}{2}$ &$n+1$ &$n+\frac{1}{2}$ & $n+1$ & $n+\frac{1}{2}$\\ \hline \hline
Label& $A^-$ &$B^-$ &$C^-$ &$D^-$ &$O^-$\\ \hline
Parity & $\Pi \C^{m}\oplus \Pi \overline{\C}^m$ & $\Pi \C^{2m}$ & $\Pi\C^{2m}$ & $\Pi \C^{2m}$ & $\C^{1|2m}$\\
$\Delta_\rho$ &$\frac{n+m+1}{2}$ &$n+m+\frac{1}{2}$ &$n+m+1$ & $n+m+\frac{1}{2}$& $n+m+1$\\ \hline
\end{tabular}
\end{table}
\renewcommand{\arraystretch}{1}
\begin{theorem}[\cite{CL1,CL2}]\label{characterization}
If $k\in \C\backslash \Q$, then the vertex superalgebra $\W^k_{X^\pm}(n,m)$ is the unique vertex superalgebra extension of $D^k_{X^\pm}(n,m)\otimes V^{k_\bg^\pm}(\bg)$ satisfying the properties (P1) and (P2).
\end{theorem}
The proofs in \cite[Theorem 9.1]{CL1} for $X=A$ and \cite[Theorem 6.3]{CL2} for $X=B,C,D,O$ only use (P2) together with the following property implied by (P1):
\begin{itemize}
    \item[(P1')]
    $\W^k_{X^\pm}(n,m)$ is an extension of the subalgebra $D_{X^\pm}^k(n,m)\otimes V^{k_\bg^\pm}(\bg)$ by modules 
    \begin{align*}
        \begin{cases}
            M_{A^\pm}^+\otimes \weyl^{k_\bg^\pm}_{\rho_\bg}(\bg)\oplus M_{A^\pm}^-\otimes \weyl^{k_\bg^\pm}_{\rho_\bg^\dagger}(\bg)& \text{if}\ X=A,\\
            M_{X^\pm}\otimes \weyl^{k_\bg^\pm}_{\rho_\bg}(\bg)& \text{if}\ X=B,C,D,O,
        \end{cases}
    \end{align*}
    whose lowest conformal weight subspace coincides with $V_{\sf nat}$ of the same parity and conformal weight as in Table \ref{tab: primary fields}. Moreover, the lower conformal weight subspaces 
    $$\bigoplus_{i=0}^{e_{\mathrm{rank}\ \ag}}M_{A^\pm,\Delta_\rho+i}^\pm,\quad \bigoplus_{i=0}^{e_{\mathrm{rank}\ \ag}}M_{X^\pm,\Delta_\rho+i}$$
    ($X=B,C,D,O$) has a PBW basis in terms of $W^2_{X^\pm(-n)}, W^i_{X^\pm(-n-1)}$, for $n\geq 0$ and $i=e_2,e_3,\cdots,e_{\mathrm{rank}\ \ag}$.
\end{itemize}
Suppose $k\in \C\backslash \Q$. 
Since $\mathbf{KL}_{k_\bg^\pm}(\bg)$ is semisimple, $\W^k_{X^\pm}(n,m)$ decomposes into
\begin{align}\label{branching rule}
\W^k_{X^\pm}(n,m)\simeq \bigoplus_{\lambda\in P_+} \mathcal{C}_{X^\pm}^k(\lambda)\otimes \weyl_{\lambda}^{k_\bg^\pm}(\bg)
\end{align}
as $C_{X^\pm}^k(n,m)\otimes V^{k_\bg^\pm}(\bg)$-modules. 
Here $\mathcal{C}_{X^\pm}^k(\lambda)$ is the $C_{X^\pm}^k(n,m)$-module consisting of highest weight vectors for $V^{k_\bg^\pm}(\bg)$ of highest weight $\lambda$ appearing in $\W^k_{X^\pm}(n,m)$.
It follows from (P1) and the character computation that 
\begin{align}\label{appearing h.wt}
{\{\lambda\in P_+\mid \mathcal{C}_{X^+}^k(\lambda) \neq 0\}=\{\lambda\in P_+\mid \mathcal{C}_{X^-}^k(\lambda) \neq 0\}=R,}
\end{align}
see \cite[Theorem 4.11, Remark 4.1]{CL1} for a proof in the case $X=A$ which holds also for $X=B,C,D,O$. 
In the case $X=A$, we use the decomposition 
$$V^{k_\bg^\pm}(\gl_m)\simeq V^{k_{\bg}^\pm}(\sll_m)\otimes \pi^\pm$$
where $\pi^\pm$ is the Heisenberg vertex algebra generated by a field $h^\pm(z)$ which corresponds to $\mathrm{Id}$ in the decomposition $\gl_m=\sll_m\oplus \C \mathrm{Id}$. It satisfies the OPE
$$h^+(z)h^+(w)\sim \frac{-m+\frac{mn}{n+m}(k+h^\vee_{A^+})}{(z-w)^2},\quad h^-(z)h^-(w)\sim \frac{m-\frac{m(n+m)}{n}(k+h^\vee_{A^-})}{(z-w)^2},$$
respectively.
\begin{proposition}\label{properties for generic case}\hspace{0mm} If $\g$ is a Lie algebra (resp.\ Lie superalgebra), then the follow hold.\\
\textup{(1) (\cite[Theorem 3.6]{CL1})}
$\W^k_{X^\pm}(n,m)$ is simple for all (resp.\ generic) $k\in \C\backslash \Q$.\\
\textup{(2)} 
The coset $C^k_{X^\pm}(n,m)$ is simple for all (resp.\ generic) $k\in \C\backslash \Q$.\\
\textup{(3) (\cite[Theorem 4.11]{CL1})}
The multiplicity space $\mathcal{C}^k_{X^\pm}(\lambda)$, ($\lambda\in R$), is simple as a $C^k_{X^\pm}(n,m)$-module for generic $k\in \C\backslash \Q$.
\end{proposition}
\noindent
We note that the second statement (2) follows from (1) by \cite[Proposition 5.4]{CGN}.

\subsection{Duality in $\W^k_{X\pm}(n,m)$}
Let $(X,Y)$ denote one of the pairs
\begin{align*}
(A,A),\quad  (B,O),\quad (C,C),\quad (D,D),\quad (O,B).
\end{align*}
The following statement, which naturally generalizes Feigin-Frenkel duality in the above types, was conjectured by Gaiotto--Rap\v{c}\'{a}k \cite{GR} in the physics literature and proven by two of the authors for generic levels.
\begin{theorem}[\cite{CL1,CL2}]\label{duality theorem}
For generic $k,\ell\in \C\backslash \Q$, there is an isomorphism of vertex algebras
\begin{align}\label{Gaiotto-rapcak duality}
     D^k_{X^+}(n,m)\simeq D^\ell_{Y^-}(n,m),
\end{align}
if the levels $(k,\ell)$ satisfy the relation
\begin{align}\label{duality rel}
    r_X(k+h^\vee_{X^+})(\ell+h^\vee_{Y^-})=1
\end{align}
with $r_X$ given in Table \ref{tab: lacing number}.
Moreover,  if $X^+=A^+$, $B^+$, $C^+$, $D^+$, then \eqref{Gaiotto-rapcak duality} holds for all $k\in \C\backslash \Q$. 
\end{theorem}
\proof
{Since the second statement is essentially obtained in the proof of \cite{CL1} when $X=A$ (resp.\  \cite{CL2} when $X=B, C, D$) for the generic case, we only sketch the proof.
By the same proof in {\it loc.\ cit}}, $D^k_{X^+}(n,m)$ and $D^\ell_{Y^-}(n,m)$ are quotients of the two-parameter $\W_\infty$-algebra $\W(c,\lambda)$ constructed in \cite{L} (resp.\ the even spin $\W_\infty$-algebra $\W^{\mathrm{ev}}(c,\lambda)$ in \cite{KL}) if $X=A$ (resp.\ $X=B,C,D,O$) with the same specialization for $(c,\lambda)$.
Then it follows that the simple quotients of $D^k_{X^+}(n,m)$ and $D^\ell_{Y^-}(n,m)$ are isomorphic as vertex algebras. 
On the other hand, by Proposition \ref{properties for generic case} (2), $C^k_{X^+}(n,m)$ is simple as a vertex algebra and thus so is $D^k_{X^+}(n,m)$ by \cite{DLM}. 
Therefore, we have a surjective homomorphism $D^\ell_{Y^-}(n,m)\twoheadrightarrow D^k_{X^+}(n,m)$ of vertex algebras, which is an isomorphism for generic $k\in \C\backslash \Q$. Since their $q$-characters are the same for all $k,\ell\in \C\backslash \Q$, $D^\ell_{Y^-}(n,m)\xrightarrow{\simeq} D^k_{X^+}(n,m)$ holds for all $k\in \C\backslash \Q$.
\endproof
\renewcommand{\arraystretch}{1.2}
\begin{table}[h]
    \caption{}
    \label{tab: lacing number}
    \centering
    \begin{tabular}{|cccccc|}
\hline
$X$ & $A$ & $B$ & $C$ & $D$ & $O$ \\ \hline
$r_X$ & 1 &1& 2&2&4\\ \hline
    \end{tabular}
\end{table}
\renewcommand{\arraystretch}{1}
In the remainder of this subsection, the statements for type $X=B, O$ with $m>1$ are understood under the assumption for the existence of the simple vertex superalgebras $\KVA{\pm1}[\so_{2m+1},k]$ (and thus $\KVA{\pm1}[\osp_{1|2m},\ell]$).
The main results of this paper is the following theorem conjectured by two of the authors \cite{CL1, CL2}:
\begin{theorem}\label{Main results}
For $k,\ell \in \C\backslash \Q$ under \eqref{duality rel}, there exist isomorphisms of vertex superalgebras
\begin{align}
  \label{relpm}  &\W^\ell_{Y^-}(n,m)\simeq \relsemicoh{0}\left(\bg, \W^k_{X^+}(n,m)\otimes \KVA{1}[\bg,\alpha_\tp]\right),\\
  \label{relmp} &\W^k_{X^+}(n,m)\simeq \relsemicoh{0}\left(\bg, \W^\ell_{Y^-}(n,m)\otimes \KVA{-1}[\bg,\alpha_\tn]\right),
\end{align}
where $(\alpha_\tp,\alpha_\tn)$ are given by 
\begin{align}\label{alpha-index}
    \alpha_\tp=-p_\tp(k+h^\vee_{X^\tp})+q_{{\scalebox{0.5}{$+$}}}-h_{\bg}^\vee,\quad \alpha_\tn=p_\tn(\ell+h^\vee_{Y^\tn})-q_\tn-h_{\bg}^\vee
\end{align}
with $(p_\tpn,q_\tpn)$ as in Table \ref{tab: pq list} below.
\end{theorem}
\renewcommand{\arraystretch}{1.2}
\begin{table}[h]
    \caption{}
    \label{tab: pq list}
    \centering
    \begin{tabular}{|cccccc|}
\hline
$X$ & $A$ & $B$ & $C$ & $D$ & $O$ \\ \hline
$(p_\tp,q_\tp)$& $(1,1)$& $(1,1)$& $(1,\tfrac{1}{2})$ & $(1,1)$ &$(1,\tfrac{1}{2})$  \\
$(p_\tn,q_\tn)$ & $(1,1)$ & $(2,1)$ & $(\tfrac{1}{2},\tfrac{1}{2})$ &$(2,1)$ &$(\tfrac{1}{2},\tfrac{1}{2})$  \\ \hline
    \end{tabular}
\end{table}
\renewcommand{\arraystretch}{1}
We will give a proof of Theorem \ref{Main results} in \S\ref{sec: Proof of Main Theorem}.
By taking the affine cosets of both sides of \eqref{relpm}, we obtain the following improvement of Theorem \ref{duality theorem}:
\begin{corollary}\label{duality theorem without orbifold}
For $k,\ell\in \C\backslash \Q$ under \eqref{duality rel}, there exists an isomorphism of vertex algebras
\begin{align*}
     C^k_{X^+}(n,m)\simeq C^\ell_{Y^-}(n,m).
\end{align*}
\end{corollary}
\noindent
Actually, this corollary will be obtained in the proof of Theorem \ref{Main results}. {We remark that $C^k_{X^+}(n,m)$ and $C^\ell_{Y^-}(n,m)$ are simple current extensions of $D^k_{X^+}(n,m)$ and $D^\ell_{Y^-}(n,m)$, respectively, and are generated as modules over $D^k_{X^+}(n,m) \simeq D^\ell_{Y^-}(n,m)$ by a field whose conformal weight and parity is the same in both cases. Another approach to proving Corollary \ref{duality theorem without orbifold} is to prove the uniqueness of this simple current extension; see \cite[Remark 4.2]{CL2}, which outlines this approach in the case $X = D$ and $m=1$.}
\begin{corollary}\label{simplicity at irrational levels}
For  $Y=A,C,D,O$, $\W^\ell_{Y^-}(n,m)$ is simple for all $\ell\in \C\backslash \Q$.
\end{corollary}
\proof
Take a nonzero ideal $\mathscr{I}$ of $\W^\ell_{Y^-}(n,m)$.
Since $\mathbf{KL}_{\ell_\bg^-}(\bg)$ is semisimple, $\mathscr{I}$ decomposes into 
$$\mathscr{I}=\bigoplus_{\lambda\in R}\mathscr{I}_\lambda\otimes \weyl_{\lambda}^{\ell_\bg^-}(\bg)$$
under \eqref{branching rule} for some $C_{Y^-}^\ell(n,m)$-submodules $\mathscr{I}_\lambda\subset \mathcal{C}_{Y^-}^\ell(\lambda)$. Applying the functor $\relsemicoh{0}\left(\bg, ? \otimes \KVA{-1}[\bg,\alpha_\tn]\right)$, we obtain a nonzero ideal 
$$\mathscr{J}=\relsemicoh{0}\left(\bg,\mathscr{I}\otimes \KVA{-1}[\bg,\alpha_\tn]\right)\simeq \bigoplus_{\lambda\in R}\mathscr{I}_\lambda\otimes \weyl_{\lambda}^{k_{\bg}^+}(\bg)$$
of $\W^k_{X^+}(n,m)$ by Theorem \ref{Main results}.
It follows from Proposition \ref{properties for generic case} (1) that $\mathscr{J}=\W^k_{X^+}(n,m)$ and thus $\mathscr{I}=\W^\ell_{Y^-}(n,m)$. 
\endproof
\subsection{Proof of Theorem \ref{Main results}}\label{sec: Proof of Main Theorem}
Let us write $\bg_\tp$ (resp.\ $\bg_\tn$) instead of $\bg$ for the $\W$-superalgebra $\W^k_{X^+}(n,m)$ (resp.\ $\W^\ell_{Y^-}(n,m)$) for clarity.
It is immediate from the equality
$${k_{\bg_\tp}^++\alpha_\tp=-2h^\vee_{\bg_\tp},\qquad \ell_{\bg_\tn}^-+\alpha_\tn=-2h^\vee_{\bg_\tn}}$$
that we may apply $\relsemicoh{0}(\bg_\tpn,?)$ in the right-hand side of \eqref{relpm} and \eqref{relmp} except for $\relsemicoh{0}(\gl_1,?)$ appearing in $\relsemicoh{0}(\gl_m,{?})$. The case $\relsemicoh{0}(\gl_1,?)$ follows from the isomorphism of the Heisenberg vertex algebras
\begin{align*}
\begin{array}{ccc}
\pi^{\ssqrt{-1}h^\pm}\otimes \pi^{h^\mp}&\simeq &\pi^{\alpha_{\ssqrt{\pm m}}}\otimes \pi^{\alpha_{\sqrt{\pm m}}}\\
\ssqrt{-1}h^{\pm}(z)\otimes 1&\mapsto & \alpha_{\ssqrt{\pm m}}(z)\otimes 1\mp u^{\pm}\otimes\alpha_{\ssqrt{\pm m}}(z)\\
1\otimes h^{\mp}(z)&\mapsto &-\alpha_{\ssqrt{\pm m}}(z)\otimes 1\mp u^{\mp}\otimes \alpha_{\ssqrt{\pm m}}(z),
\end{array}
\end{align*}
where $u=\sqrt{\tfrac{-n}{m+n}(k+h^\vee_{A^+})}$ and thus $u^{-1}=\sqrt{\tfrac{m+n}{-n}(\ell+h^\vee_{A^-})}$.
Here $\alpha_{\sqrt{\pm m}}(z)$ is the Heisenberg field satisfying the OPE
$\alpha_{\sqrt{\pm m}}(z)\alpha_{\sqrt{\pm m}}(w)\sim \pm m/(z-w)^2$.
It induces isomorphisms of their Fock modules
\begin{align*}
\pi^{\alpha_{\sqrt{\pm m}}}_a\otimes \pi^{\alpha_{\sqrt{\pm m}}} \simeq \pi^{\ssqrt{-1}h^{\pm}}_{a}\otimes \pi^{h^\mp}_{-a},\quad (a\in \C).
\end{align*}
{By using Theorem \ref{main theorem of relcoh}, we obtain, case by case, isomorphisms
\begin{align*}
\mathscr{W}^\ell_{Y^-}&:=\relsemicoh{0}\left(\bg_\tp, \W^k_{X^+}(n,m)\otimes \KVA{1}[\bg_\tp,\alpha_\tp]\right)
 \simeq \bigoplus_{\lambda\in R}\mathcal{C}_{X^+}^k(\lambda)\otimes \weyl_{\op{\lambda}}^{\ell^-_{\bg_\tn}}(\bg_\tn),\\
\mathscr{W}^k_{X^+}&:=\relsemicoh{0}\left(\bg_\tn, \W^\ell_{Y^-}(n,m)\otimes \KVA{-1}[\bg_\tn,\alpha_\tn]\right)
\simeq 
\bigoplus_{\lambda\in R}\mathcal{C}_{Y^-}^\ell(\lambda)\otimes \weyl_{\op{\lambda}}^{k^+_{\bg_\tp}}(\bg_\tp),
\end{align*}
as modules over $C^k_{X^+}(n,m)\otimes V^{\ell^-_{\bg_\tn}}(\bg_\tn)$ and $C^k_{Y^-}(n,m)\otimes V^{k^+_{\bg_\tp}}(\bg_\tp)$, respectively.
Here, we have used
\begin{align*}
    V^{\op{\alpha}_\tp}(\op{\bg}_\tp)\simeq V^{\ell^-_{\bg_\tn}}(\bg_\tn),
    \qquad V^{\op{\alpha}_\tn}(\op{\bg}_\tn)\simeq V^{k^+_{\bg_\tp}}(\bg_\tp),
\end{align*}
which can be checked by using the relations \eqref{final relations}, \eqref{duality rel} and \eqref{alpha-index}.}
By using the duality $D^k_{X^+}(n,m) \simeq D^\ell_{Y^-}(n,m)$ (Theorem \ref{duality theorem}), we regard the vertex superalgebra $\mathscr{W}^\ell_{Y^-}$ as a module over $D^\ell_{Y^-}(n,m)\otimes V^{\ell^-_{\bg_\tn}}(\bg_\tn)$ and $\mathscr{W}^k_{X^+}$ over $D^\ell_{X^+}(n,m)\otimes V^{k^+_{\bg_\tp}}(\bg_\tp)$, respectively. 

{We show that $\mathscr{W}^\ell_{Y^-}$ and $\mathscr{W}^k_{X^+}$ contain some vertex subalgebras which satisfy the characterizations of $\W_{Y^-}^\ell(n,m)$ and $\W_{X^+}^k(n,m)$, respectively, namely (P1') and (P2) by Theorem \ref{characterization} and the discussion after it.}
For this purpose, we construct the subspace $V_{\mathrm{nat}}$ in (P2) for both cases. By Theorem \ref{main theorem of relcoh}, we may associate with $A_\tnp\in \rho_{\bg_\tnp}$ or $\rho_{\bg_\tnp}^\dagger$ the following (nonzero) cohomology class
$$\Big[ \str_{\rho_{\bg_\tpn}}\otimes A_\tnp \Big],\quad \Big[ \str_{\rho_{\bg_\tpn}^\dagger}\otimes A_\tnp\Big],$$
respectively. Let $\widehat{A}_\tnp$ denote the above cohomology class.
By using Table \ref{tab: abc} and \ref{tab: primary fields}, we find case by case that the parity and the conformal dimension of $\widehat{A}_\tnp$ agree with those of $S_{\sf nat}$ and $S_{\sf nat}^\dagger$ for $\W_{Y^-}^\ell(n,m)$ and $\W_{X^+}^k(n,m)$, respectively.
Then it follows from Corollary \ref{nondegeneracy under relcoh} that the pairing for the cohomology classes $\widehat{A}_\tnp$, ($A_\tnp \in \rho_{\bg_\tnp}$, $\rho_{\bg_\tnp}^\dagger$), is non-degenerate by Theorem \ref{characterization} and Lemma \ref{tentative}. Hence, the subspace $[V_{\mathrm{nat}}]$ spanned by $A_\tnp$'s satisfies the condition (P2). {Now, let us introduce the following vertex subalgebras by specifying the generating subspaces:
\begin{align*}
    &\mathscr{D}^\ell_{Y^-}:=\langle D^\ell_{Y^-}(n,m)\otimes V^{\ell^-_{\bg_\tn}}(\bg_\tn), [V_{\mathrm{nat}}]\rangle \subset \mathscr{W}^\ell_{Y^-},\\
    &\mathscr{D}^k_{X^+}:=\langle D^k_{X^+}(n,m)\otimes V^{\ell^+_{\bg_\tp}}(\bg_\tp), [V_{\mathrm{nat}}]\rangle \subset \mathscr{W}^k_{X^+},
\end{align*}
which decompose into 
\begin{align*}
\mathscr{D}^\ell_{Y^-}
 \simeq \bigoplus_{\lambda\in R}\mathcal{D}^\ell_{Y^-}(\lambda)\otimes \weyl_{\op{\lambda}}^{\ell^-_{\bg_\tn}}(\bg_\tn),\quad 
\mathscr{D}^k_{X^+}
\simeq \bigoplus_{\lambda\in R}\mathcal{D}^k_{X^+}(\lambda)\otimes \weyl_{\op{\lambda}}^{k^+_{\bg_\tp}}(\bg_\tp)
\end{align*}
as modules over $D^\ell_{Y^-}(n,m)\otimes V^{\ell^-_{\bg_\tn}}(\bg_\tn)$ (resp.\ $D^k_{X^+}(n,m)\otimes V^{\ell^+_{\bg_\tp}}(\bg_\tp)$) for some $D^\ell_{Y^-}(n,m)$-submodules $\mathcal{D}_{Y^-}^\ell(\lambda)\subset \mathcal{C}_{X^+}^k(\lambda)$ (resp.\  $D^k_{X^+}(n,m)$-submodules $\mathcal{D}_{X^+}^k(\lambda)\subset \mathcal{C}_{X^+}^k(\lambda)$).}

{We first show the isomorphism $\mathscr{D}^k_{X^+}\simeq \W^k_{X^+}(n,m)$.
Note that it suffices to establish (P1)' in our setting. 
By construction, $\mathscr{D}^k_{X^+}$ is an extension of $D^k_{X^+}(n,m)\otimes V^{\ell^+_{\bg_\tp}}(\bg_\tp)$ by the module
\begin{align*}
\begin{cases}
     \mathcal{D}^k_{X^+}(\op{\rho}_{\bg_\tp})\otimes \weyl_{\rho_{\bg_\tp}}^{k^+_{\bg_\tp}}(\bg_\tp) \oplus \mathcal{D}^k_{X^+}(\op{\rho}_{\bg_\tp}^\dagger)\otimes \weyl_{\rho_{\bg_\tp}^\dagger}^{k^+_{\bg_\tp}}(\bg_\tp) & (X=A),\\
     \mathcal{D}^k_{X^+}(\op{\rho}_{\bg_\tp})\otimes \weyl_{\rho_{\bg_\tp}}^{k^+_{\bg_\tp}}(\bg_\tp) & (X=B,C,D,O).
\end{cases}
\end{align*}
Note that $\mathcal{D}^k_{X^+}(\op{\rho}_{\bg_\tp})$ (and $\mathcal{D}^k_{X^+}(\op{\rho}_{\bg_\tp}^\dagger)$ for $X=A$) satisfies (P1') for $\W^\ell_{Y^-}(n,m)$, that is, the PBW basis property with respect to $D^\ell_{Y^-}(n,m)$ up to the conformal weight $e_{\mathrm{rank}\ag}$. 
We identify the strong generators $W^i_{X^+}$ for $D^k_{X^+}(n,m)$ with $W^i_{Y^-}$ for $D^\ell_{Y^-}(n,m)$ through the duality $D^k_{X^+}(n,m) \simeq D^\ell_{Y^-}(n,m)$ by Theorem \ref{duality theorem}.
On the other hand, the number $e_{\mathrm{rank}\ \ag}$ of $\W^\ell_{Y^-}(n,m)$ is bigger or equal to that of $\W^k_{X^+}(n,m)$ by Table \ref{tab: hopok-type Walg}.
Then, it follows that $\mathcal{D}^k_{X^+}(\op{\rho}_{\bg_\tp})$ (and $\mathcal{D}^k_{X^+}(\op{\rho}_{\bg_\tp}^\dagger)$ for $X=A$) satisfies the PBW basis property (P1') with respect to $D^k_{X^+}(n,m)$, i.e., (P1') for $\W^k_{X^+}(n,m)$. Therefore, we have a surjection $\W^k_{X^+}(n,m)\twoheadrightarrow \mathscr{D}^k_{X^+}$, which is an isomorphism by Proposition \ref{properties for generic case} (1).} Hence, we obtain an embedding 
\begin{align}\label{hom from non to relcoh}
\W^k_{X^+}(n,m)\hookrightarrow \relsemicoh{0}\left(\bg_\tn, \W^\ell_{Y^-}(n,m)\otimes \KVA{-1}[\bg_\tn,\alpha_\tn]\right).
\end{align}
and thus embeddings $\mathcal{C}^k_{X^+}(\lambda)\hookrightarrow \mathcal{C}^\ell_{Y^-}(\op{\lambda})$ of $D^k_{X^+}(n,m)$-modules for all $\lambda\in R$. In particular, we have 
$C^k_{X^+}(n,m)\hookrightarrow C^\ell_{Y^-}(n,m)$. In the case $X=A,C$, this is indeed an isomorphism by Theorem \ref{duality theorem}. In the case $X=B,C,O$, these coset algebras are second-order extensions of the same algebra $D^k_{X^+}(n,m)\simeq D^\ell_{Y^-}(n,m)$ by Theorem \ref{duality theorem} and therefore $C^k_{X^+}(n,m)\simeq C^\ell_{Y^-}(n,m)$ as well.
Now, if $k\in \C\backslash \Q$ is generic, then it follows from Proposition \ref{properties for generic case} (3) that $\mathcal{C}^k_{X^+}(\lambda)\hookrightarrow \mathcal{C}^\ell_{Y^-}(\op{\lambda})$ is indeed an isomorphism of $C^k_{X^+}(n,m)$-modules.
Hence, the $q$-characters of $\mathcal{C}^k_{X^+}(\lambda)$ and $\mathcal{C}^\ell_{Y^-}(\op{\lambda})$ coincide.
Since the $q$-characters are all equal for $k,\ell\in \C\backslash \Q$, $\mathcal{C}^k_{X^+}(\lambda)\simeq \mathcal{C}^\ell_{Y^-}(\op{\lambda})$ holds for all $k\in \C\backslash \Q$. 
Therefore, \eqref{hom from non to relcoh} is an isomorphism of vertex algebras for all $k\in \C\backslash \Q$. This completes the proof of \eqref{relmp}. 
In particular, it follows that $\mathcal{D}^k_{X^+}(\rho_{\bg_\tp})$ (and $\mathcal{D}^k_{X^+}(\rho^\dagger_{\bg_\tp})$ for $X=A$) satisfies (P1') for $\W_{Y^-}^\ell(n,m)$.
As we have already shown the property (P2), we may apply Theorem \ref{characterization} and obtain a surjection  
\begin{align*}
\W^\ell_{Y^-}(n,m)\twoheadrightarrow \mathscr{D}^\ell_{Y^-}\subset\relsemicoh{0}\left(\bg_\tp, \W^k_{X^+}(n,m)\otimes \KVA{1}[\bg_\tp,\alpha_\tp]\right).
\end{align*}
It follows from \eqref{relmp} and Theorem \ref{main theorem of relcoh} that it is an isomorphism as $D^\ell_{Y^-}(n,m)\otimes V^{\ell_{\bg_\tn}^-}(\bg_\tn)$-modules and thus as vertex superalgebras. This completes the proof of \eqref{relpm}.

\appendix \label{appendix}
\section{Proof of Theorem \ref{main theorem of relcoh}}\label{proof of formality of relcoh}
In this appendix, we prove Theorem \ref{main theorem of relcoh} when $\g$ is simple. We set $\kappa_1=k \kappa_0$, $\kappa_2=\ell \kappa_0$ by using the normalized invariant bilinear form $\kappa_0$ and thus $k,\ell\in \C\backslash \Q$.  We will use the decompositions of the loop algebra $\mathrm{L}\g=\g[t^{\pm1}]=\lpp{+}\g\oplus\g[t^{-1}]=\g[t]\oplus \lpp{-}\g$ with $\lpp{\pm}\g=\g[t^{\pm1}]t^{\pm1}$.
\subsection{Spectral sequences}\label{sec: spectral sequences}
In order to prove Theorem \ref{main theorem of relcoh}, we consider spectral sequences in the general setting. 
Given an object $M=\bigoplus_{\Delta\in \Z}M_{\Delta}$ in $\KL_{-2h^\vee}$, let us write
$\mathscr{C}^\bullet(M)=C^{\frac{\infty}{2}+\bullet}(\g,M)$ (resp.\ $\mathscr{C}^\bullet_{\mathrm{rel}}(M)=C^{\frac{\infty}{2}+\bullet}_{\mathrm{rel}}(\g,M)$) and introduce on it a filtration
$F_\bullet \mathscr{C}^\bullet(M)$ in the following way:
we define a grading $\mathscr{C}^\bullet(M)=\bigoplus_{\Delta_F\in\Z}\mathscr{C}^\bullet_{\Delta_F}(M)$ by setting
\begin{align}\label{F-grading}
\Delta_F (X_{n})=n,\quad \Delta_F (\varphi_{i,n})=|n|,\quad \Delta_F (\varphi^*_{i,n})=-|n|,\quad \Delta_F|_{M_\Delta}=-\Delta
\end{align}
and then 
$$F_p \mathscr{C}^\bullet(M)=\bigoplus_{\Delta_F\leq -p}\mathscr{C}^\bullet_{\Delta_F}(M),\quad (p\in \Z).$$
The filtration $F_\bullet \mathscr{C}^\bullet(M)$ is decreasing, exhaustive and separated, i.e. satisfies 
$F_p \mathscr{C}^\bullet(M)\supset F_{p+1}\mathscr{C}^\bullet(M)$, 
$\cup_p F_p \mathscr{C}^\bullet(M)=\mathscr{C}^\bullet(M)$ and $\cap_pF_p \mathscr{C}^\bullet(M)=0$, respectively.
Since the differential $d$ preserves the filtration $d\colon F_p \mathscr{C}^\bullet(M)\rightarrow F_p \mathscr{C}^\bullet(M)$, we have the associated spectral sequence 
\begin{align}
E_1^F=H^\bullet(\gr_F\mathscr{C}^\bullet(M),\gr_Fd)\Rightarrow E_\infty^F=\gr_F \semicoh{\bullet}(\g,M).
\end{align}
Note that the spectral sequence converges by Remark \ref{decomposition of complexes}.
(Indeed, the filtrations are finite when restricted to the subcomplexes in Remark \ref{decomposition of complexes} thanks to $\Delta_F (\varphi^*_{i})=0$.)
The induced differential decomposes into $\gr_Fd=d_1^F+d_2^F$ with 
\begin{align}\label{F-graded d}
\begin{split}
&d_1^F=\sum_{n\geq0}(-1)^{\bar{x}_i}x_{i(n)}\varphi_{i,-n}^*-\frac{1}{2}\sum_{p,q\geq0}(-1)^{\bar{x}_i\bar{x}_k}c_{i,j}^k\varphi_{i,-p}^*\varphi_{j,-q}^*\frac{\partial}{\partial \varphi_{k,-p-q}^*},\\
&d_2^F=-\frac{1}{2}\sum_{p,q>0}(-1)^{\bar{x}_i\bar{x}_k+\bar{x}_i+\bar{x}_j}c_{i,j}^k\varphi_{k,-p-q}\frac{\partial}{\partial \varphi_{i,-p}}\frac{\partial}{\partial \varphi_{j,-q}}.
\end{split}
\end{align}
Then it follows that the complex $(\gr_F\mathscr{C}^\bullet(M),\gr_Fd)$ decomposes into the tensor product
$$(\gr_F\mathscr{C}^\bullet(M),\gr_Fd)\simeq \left(M\otimes \bigwedge\hspace{-0.5mm}_{\hspace{1mm}\partial^n\varphi_i^*},d_1^F\right)\otimes \left(\C\otimes \bigwedge\hspace{-0.5mm}_{\hspace{1mm}\partial^n\varphi_i},d_2^F\right).$$
and that the first component in the right-hand side is nothing but the Chevalley--Eilenberg cochain complex of $\g[t]$ with coefficients in $M$ and the second component the chain complex of $\lpp{-}\g$ with coefficients in $\C$. Therefore, $E_1^F\simeq H^\bullet(\g[t],M)\otimes H_\bullet(\lpp{-}\g,\C)$.

By restricting $F_\bullet \mathscr{C}^\bullet(M)$ to $\mathscr{C}^\bullet_{\mathrm{rel}}(M)$, namely, $F_p\mathscr{C}^\bullet_{\mathrm{rel}}(M)=\mathscr{C}^\bullet_{\mathrm{rel}}(M) \cap F_p\mathscr{C}^\bullet(M)$, we obtain a similar spectral sequence 
$$E_{1,\mathrm{rel}}^F=H^\bullet(\gr_F\mathscr{C}^\bullet_{\mathrm{rel}}(M),\gr_Fd)\Rightarrow E_{\infty,\mathrm{rel}}^F=\gr_F \relsemicoh{\bullet}(\g,M)$$
with $\gr_Fd=d_{1,\mathrm{rel}}^{F}+d_2^F$ where
$$d_{1,\mathrm{rel}}^F=\sum_{n>0}(-1)^{\bar{x}_i}x_{i,n}\varphi_{i,-n}^*-\frac{1}{2}\sum_{p,q>0}(-1)^{\bar{x}_i\bar{x}_k}c_{i,j}^k\varphi_{i,-p}^*\varphi_{j,-q}^*\frac{\partial}{\partial \varphi_{k,-p-q}^*}.$$
In this case, we have 
\begin{align*}
E_{1,\mathrm{rel}}^F
&=H^\bullet\left(\left(\left(M\otimes \bigwedge\hspace{-0.5mm}_{\hspace{1mm}\partial^{n+1}\varphi_i^*},d_{1,\mathrm{rel}}^F\right)\otimes \left(\C\otimes \bigwedge\hspace{-0.5mm}_{\hspace{1mm}\partial^n\varphi_i},d_2^F\right)\right)^\g\right)\\
&\simeq H^\bullet\left(\left(M\otimes \bigwedge\hspace{-0.5mm}_{\hspace{1mm}\partial^{n+1}\varphi_i^*},d_{1,\mathrm{rel}}^F\right)\otimes \left(\C\otimes \bigwedge\hspace{-0.5mm}_{\hspace{1mm}\partial^n\varphi_i},d_2^F\right)\right)^\g\\
&=\left(H^\bullet(\lpp{+}\g,M)\otimes H_\bullet(\lpp{-}\g,\C)\right)^\g.
\end{align*}
Here in the second line, we have used the commutativity of taking cohomology and $\g$-invariant subspace, which follows from the complete reducibility of integrable $\g$-modules.

Next, we introduce another filtration $G_\bullet \mathscr{C}^\bullet(M)$ on $\mathscr{C}^\bullet(M)$ in the following way: we define a grading $\Delta_G$ in the following way: we replace \eqref{F-grading} with 
$$\Delta_G (X_{n})=n,\quad \Delta_G (\varphi_{i,n})=-|n|,\quad \Delta_G (\varphi^*_{i,n})=|n|,\quad \Delta_G|_{M_\Delta}=-\Delta$$
and set $G_\bullet \mathscr{C}^\bullet(M)$ by 
$$G_p \mathscr{C}^\bullet(M)=\bigoplus_{\Delta_G\geq p}\mathscr{C}^\bullet_{\Delta_G}(M),\quad (p\in \Z).$$
Again, the filtration $G_\bullet \mathscr{C}^\bullet(M)$ is decreasing, exhaustive, and separated. Since $d$ preserves the filtration, we have the associated spectral sequence 
$$E_1^G=H^\bullet(\gr_G\mathscr{C}^\bullet(M),\gr_Gd)\Rightarrow E_\infty^G=\gr_G \semicoh{\bullet}(\g,M)$$
with $\gr_Gd=d_1^G+d_2^G$ where 
\begin{align}\label{G-graded d}
\begin{split}
&d_1^G=\sum_{n\geq0}(-1)^{\bar{x}_i}x_{i,-n}\varphi_{i,n}^*-\frac{1}{2}\sum_{p,q\geq0}(-1)^{\bar{x}_i\bar{x}_k}c_{i,j}^k\NO{\varphi_{i,p}^*\varphi_{j,q}^* \varphi_{k,-p-q}},\\
&d_2^G=-\frac{1}{2}\sum_{p,q>0}(-1)^{\bar{x}_i\bar{x}_k}c_{i,j}^k\varphi_{i,-p}^*\varphi_{j,-q}^* \frac{\partial}{\partial \varphi^*_{k,-p-q}}.
\end{split}
\end{align}
In this case, the complex $(\gr_G\mathscr{C}^\bullet(M),\gr_Gd)$ decomposes into
$$(\gr_G\mathscr{C}^\bullet(M),\gr_Gd)\simeq \left(M\otimes \bigwedge\hspace{-0.5mm}_{\hspace{1mm}\varphi_i^*,\partial^n\varphi_i},d_1^G\right)\otimes \left(\C\otimes \bigwedge\hspace{-0.5mm}_{\hspace{1mm}\partial^{n+1}\varphi_i^*},d_2^G\right)$$
and the second component in the right-hand side is the cochain complex calculating $H^\bullet(\lpp{+}\g,\C)$.
Although the first component lacks a clear meaning, by restricting to
$$\mathrm{gr}_G\mathscr{C}^\bullet_{\mathrm{rel}}(M)=\left(\left(M\otimes \bigwedge\hspace{-0.5mm}_{\hspace{1mm}\partial^{n+1}\varphi_i^*}\right)\otimes \left(\C\otimes \bigwedge\hspace{-0.5mm}_{\hspace{1mm}\partial^n\varphi_i}\right)\right)^\g$$
we find that the induced differential $d_1^G$ acts on the first component by
\begin{align*}
d_{1,\mathrm{rel}}^G
=\sum_{n>0}x_{i,-n}\frac{\partial}{\partial \varphi_{i,-n}}-\frac{1}{2}\sum_{p,q>0}(-1)^{\bar{x}_i\bar{x}_k+\bar{x}_i+\bar{x}_j}c_{i,j}^k  \varphi_{k,-p-q} \frac{\partial}{\partial \varphi_{i,-p}}\frac{\partial}{\partial \varphi_{j,-q}}.
\end{align*}
Thus, the first component in the right-hand side is the chain complex calculating $H_\bullet(\lp^-\g, M)$. Therefore, we have obtained a convergent spectral sequence with
\begin{align*}
E_{1,\mathrm{rel}}^G\simeq \left(H_\bullet(\lpp{-}\g,M)\otimes H^\bullet(\lpp{+}\g,\C)\right)^\g.
\end{align*}
To summarize, we have obtained the following proposition.
\begin{proposition}\label{spectral seq for relcoh}
For an object $M$ in $\KL_{-2h^\vee}$, we have convergent spectral sequences 
\begin{align*}
&E_{1,\mathrm{rel}}^F\simeq \left(H^\bullet(\lpp{+}\g,M)\otimes H_\bullet(\lpp{-}\g,\C)\right)^\g\Rightarrow E_\infty^F=\gr_F \relsemicoh{\bullet}(\g,M),\\
&E_{1,\mathrm{rel}}^G\simeq \left(H_\bullet(\lpp{-}\g,M)\otimes H^\bullet(\lpp{+}\g,\C)\right)^\g\Rightarrow E_\infty^G=\gr_G \relsemicoh{\bullet}(\g,M).
\end{align*}
In particular, if $H^n(\lpp{+}\g,M)=0$, $H_n(\lpp{-}\g,M)=0$, ($n>0$), then $\relsemicoh{n}(\g,M)=0$ holds for $n\neq0$.
\end{proposition}

\subsection{Pairings between cohomology and homology}\label{pairing}
We make a digression for pairings between Lie superalgebra cohomology and homology
\begin{align}\label{pairing at cohomology}
H^n(\g_1,M_1)\otimes H_n(\g_2,M_2)\rightarrow \C
\end{align}
for (different) Lie superalgebras $\g_i$ and their modules $M_i$ ($i=1,2$) under suitable conditions, which we will use Proposition \ref{cohomology vanishing} below.
Let $\g$ be an arbitrary Lie superalgebra and $M$ a $\g$-module.
The cochain complex of $\g$ with coefficients in $M$ is defined as
$$C^\bullet(\g,M)=\bigoplus_{n\in \Z_{\geq0}}C^n(\g,M),\quad C^n(\g,M)=\Hom_\C(\mathrm{Sym}^n(\Pi\g),M),$$
equipped with $d_n\colon C^n(\g,M)\rightarrow C^{n+1}(\g,M)$ given by 
\begin{align*}
&d_nf(x_1 x_2\cdots x_{n+1})
=\sum_{1\leq i\leq n+1}(-1)^{A_i}x_if(x_1\cdots\hat{x}_i \cdots x_{n+1})\\
&\hspace{4cm}+\sum_{1\leq i<j\leq n+1}(-1)^{A_{i,j}}f([x_i,x_j] x_1\cdots\hat{x}_i \cdots \hat{x}_j \cdots x_{n+1}),
\end{align*}
where
\begin{align*}
    &A_i=\bar{x}_i+(\bar{f}+k_{i-1})(\bar{x}_i+1),\\
    &A_{i,j}=\bar{f}+(k_{i-1}+1)(\bar{x}_i+1)+(k_{j-1}+\bar{x}_i+1)(\bar{x}_j+1),
\end{align*}
with
$k_i=\overline{x}_{1}+\cdots+\overline{x}_{i}+i$.
Similarly, the chain complex of $\g$ with coefficients in $M$ is defined as
$$C_\bullet(\g,M)=\bigoplus_{n\in \Z_{\geq0}}C_n(\g,M),\quad C_n(\g,M)=M\otimes \mathrm{Sym}(\Pi\g),$$
equipped with $\partial_n\colon C_n(\g;M)\rightarrow C_{n-1}(\g;M)$ given by 
\begin{align*}
&\partial_n(\alpha\otimes x_1 x_2\cdots x_n)
=\sum_{1\leq i\leq n}(-1)^{B_i}(x_i\alpha)\otimes x_1\cdots\hat{x}_i \cdots x_n\\
&\hspace{4cm}+\sum_{1\leq i<j\leq n}(-1)^{B_{i,j}}\alpha\otimes [x_i,x_j] x_1\cdots\hat{x}_i \cdots \hat{x}_j \cdots x_n
\end{align*}
where 
\begin{align*}
&B_i=(\bar{\alpha}+k_{i-1})(\bar{x}_i+1),\\ &B_{i,j}=1+\bar{\alpha}+(k_{i-1}+1)(\bar{x}_i+1)+(k_{j-1}+\bar{x}_i+1)(\bar{x}_j+1).
\end{align*}
\begin{remark}\textup{
The identification of the cochain (resp.\ chain) complex given above and the one in the vertex algebraic formulation appearing in \S \ref{sec: spectral sequences} is as follows: 
\begin{align*}
\begin{array}{ccl}
M\otimes \wedge_{\varphi_i}&\xrightarrow{\simeq} &\Hom_\C(\mathrm{Sym}(\Pi\g),M)\\
m\otimes \varphi_{i_1}\cdots \varphi_{i_n}&\mapsto& \left[x_{j_1}\cdots x_{j_n}\mapsto \sum_{\sigma\in S_n}(-1)^\bullet \delta_{(i_1,\cdots i_n),(j_1,\cdots, j_n)}m\right],\\
M\otimes \wedge_{\varphi_i}&\xrightarrow{\simeq}&M\otimes \mathrm{Sym}(\Pi\g)\\
m\otimes \varphi_{i_1}\cdots \varphi_{i_n}&\mapsto& m\otimes x_{i_1}\cdots x_{i_n}.
\end{array}
\end{align*}
Here $S_n$ is the $n$-th symmetric group and the sign $(-1)^\bullet$ is assigned as
$$\varphi_{i_1}\varphi_{i_2}\colon x_{j_1}x_{j_2}\mapsto \delta_{i_2,j_1}\delta_{i_1,j_2}+(-1)^{(\bar{x}_{i_2}+1)(\bar{x}_{j_1}+1)}\delta_{i_1,j_1}\delta_{i_2,j_2}.$$}
\end{remark}
Given even linear maps 
$$^t\colon \g_2\rightarrow \g_1,\quad \Psi\colon M_1\otimes M_2\rightarrow \C,$$
we may extend the pairing $\Psi$ to the even pairing $\Psi_\bullet$ on the whole complex by
\begin{align*}
\begin{array}{cccl}
\Psi_n\colon & C^n(\g_1,M_1)\otimes C_n(\g_2,M_2)& \rightarrow &\C \\
&  f\otimes \left(\alpha \otimes x_1\cdots x_n\right)&\mapsto&
(-1)^{\bar{x}_1+\cdots \bar{x}_n}\Psi\left(f\left(^tx_n\cdots \ ^tx_1\right)\otimes \alpha\right).
\end{array}
\end{align*}
\begin{lemma}
If $\Psi_\bullet$ satisfies 
\begin{align}\label{compatibility with differentials}
\Psi_{n+1}(d_nf\otimes P)=\Psi_n(f\otimes \partial_{n+1} P),
\end{align} then $\Psi$ induces an even pairing 
$[\Psi_\bullet]\colon H^\bullet(\g_1,M_1)\otimes H_\bullet(\g_2,M_2)\rightarrow \C$.
Moreover, if $\Psi_\bullet$ is non-degenerate and $\iota$ is an isomorphism, then so is $[\Psi_\bullet]$.
\end{lemma}
This is immediate from the following statement in linear algebra:
\begin{lemma}
Let $S_i$ and $T_i$ $(i=1,2)$ be vector superspaces (over $\C$) equipped with odd homomorphisms
$f\colon S_1\rightarrow S_2,\quad g\colon T_2\rightarrow T_1$
and even pairings
$(\cdot,\cdot)_i\colon S_i\otimes T_i\rightarrow \C$ $(i=1,2)$ such that 
$(f\alpha,\beta)_2=(\alpha,g\beta)_1$ for $\alpha\in S_1$, $\beta\in T_2$.
Then the pairings induce even pairings 
$$\langle\cdot,\cdot\rangle_1\colon \Ker f\otimes \Coker g\rightarrow \C,\quad \langle\cdot,\cdot\rangle_2\colon \Coker f\otimes \Ker g\rightarrow \C.$$
If $(\cdot,\cdot)_i$ are non-degenerate, then so are $\langle\cdot,\cdot\rangle_i$.
\end{lemma}
\proof
Straightforward.
\endproof

\begin{proposition}\label{condition for pairing}
Let $\g_i$ ($i=1,2$) be Lie superalgebras and $M_i$ ($i=1,2$) $\g_i$-modules. Suppose that we have an even linear isomorphism $\iota\colon\g_2\xrightarrow{\simeq} \g_1$ and a non-degenerate even pairing $\Psi\colon M_1\otimes M_2\rightarrow \C$ (as vector superspaces). Then 
the condition \eqref{compatibility with differentials} holds if and only if 
\begin{enumerate}
\item[(Q1)] $^t$ is an anti-isomorphism of Lie superalgebras,
\item[(Q2)] $\Psi$ satisfies $\Psi(m_1\otimes x.m_2)=\Psi( ^t x. m_1\otimes m_2)$ ($x\in \g_2,\ m_i\in M_i$).
\end{enumerate}
\end{proposition}
\proof
It is straightforward to see that \eqref{compatibility with differentials} for $n=0$ is (Q2) and then that \eqref{compatibility with differentials} for $n=1$ is (Q1). We prove that (Q1) and (Q2) are sufficient. Suppose $n\geq2$ and take
$f\in \Hom_\C(\mathrm{Sym}^n(\Pi \g_1),M_1)$ and $P=\alpha\otimes x_1 x_2\cdots x_{n+1} \in M_2\otimes \mathrm{Sym}^{n+1}(\Pi \g_2)$. 
Firstly, we have 
$$\Psi_{n+1}(d_nf,P)=(-1)^\bigstar\Psi(d_nf(\ ^tx_1\cdots \tran x_{n+1}),\alpha)$$
with $\bigstar=\sum_p\bar{x}_p+\sum_p(\bar{x}_{p}+1)k_{p-1}$. Then, the right-hand side is equal to
\begin{align*}
&\sum_{i=1}^{n+1}(-1)^{\bigstar+A_i}\Psi\left( \tran x_if( \tran x_1\cdots \tran\widehat{x}_i\cdots \tran x_{n+1}),\alpha\right)\\
&\hspace{5mm}+\sum_{1\leq i<j\leq n+1}(-1)^{\bigstar+A_{i,j}}\Psi\left(f([\tran x_i,\tran x_j] \tran x_1 \cdots \tran\widehat{x}_i\cdots \tran\widehat{x}_j\cdots \tran x_{n+1}),\alpha\right)\\
&=\sum_{i=1}^{n+1}(-1)^{\bigstar+A_i}\Psi(f(\tran x_1\cdots \tran\widehat{x}_i\cdots \tran x_{n+1}),x_i\alpha)\\
&\hspace{5mm}+\sum_{1\leq i<j\leq n+1}(-1)^{\bigstar+A_{i,j}+\bar{x}_i\bar{x}_j+1}\Psi\left(f(\tran [x_i,x_j] \tran x_1 \cdots \tran\widehat{x}_i\cdots \tran\widehat{x}_j\cdots \tran x_{n+1}),\alpha\right)\\
&=\sum_{i=1}^{n+1}(-1)^{A_i+\bar{x}_i+\sum_p(\bar{x}_{p}+1)k_{p-1}}\Psi(f,x_i\alpha\otimes x_{n+1}\cdots \widehat{x}_i\cdots x_1)\\
&\hspace{5mm}+\sum_{1\leq i<j\leq n+1}(-1)^{\bigstar+A_{i,j}+\bar{x}_i\bar{x}_j+1}\Psi\left(f,\alpha\otimes x_{n+1}\cdots \widehat{x}_j \cdots \widehat{x}_i\cdots x_1[x_i,x_j]\right)\\
&=\sum_{i=1}^{n+1}(-1)^{B_i}\Psi(f,x_i\alpha\otimes x_1\cdots \widehat{ x}_i\cdots x_{n+1})\\
&\hspace{5mm}+\sum_{1\leq i<j\leq n+1}(-1)^{B_{i,j}}\Psi\left(f,\alpha\otimes [x_i,x_j] x_1 \cdots\widehat{x}_i\cdots \widehat{x}_j\cdots x_{n+1}\right)\\
&=\Psi_n(f,\partial_{n+1}P).
\end{align*}
This completes the proof.
\endproof

\subsection{Proof of Theorem \ref{main theorem of relcoh}}
The following is well-known.
\begin{proposition}[e.g.\ \cite{Fuks}]
For an arbitrary Lie superalgebra $\g$, and a free $U(\g)$-module $M$, we have $$H_n(\g,M)\simeq \delta_{n,0}M/\g M.$$
\end{proposition}
\proof 
We include a sketch of proof for the completeness of the paper. 
Since free modules are inductive limits of direct sums of $U(\g)$ and inductive limit commutes with taking cohomology, we may assume $M=U(\g)$ from the beginning.
Then it is straightforward to see $H_0(\g,U(\g))\simeq U(\g)/\g U(\g)(\simeq \C)$ and thus it suffices to show $H_n(\g,U(\g))=0$ for $n\neq0$. 
By using a PBW filtration $U_{\leq n}(\g)$, we define the following decreasing filtration on the complex  
$$F_p C_\bullet (\g,U(\g))=\bigoplus_{n+m=p}U_{\leq n}(\g)\otimes \mathrm{Sym}^m(\Pi\g).$$
The filtration $F_\bullet C_{\bullet} (\g, U(\g))$ is exhaustive, separated and preserved by the differential. Thus, we have the associated spectral sequence 
$$E_1=H_\bullet (\gr_F C_\bullet(\g,U(\g)),\gr_F\partial)\Rightarrow E_\infty=\gr_F H_\bullet(\g,U(\g))$$
which converges since each $C_n(\g,U(\g))$ is finitely filtered.
Now, $\gr_F C_\bullet(\g, U(\g))= \mathrm{Sym}(\g)\otimes \mathrm{Sym}(\Pi\g)$ with 
$$\gr_F\partial\colon \alpha\otimes x_1\cdots x_n\mapsto \sum_{i=1}^n (-1)^{B_i}x_i\alpha \otimes x_1 \cdots \hat{x}_i\cdots x_n.$$
Then, it is straightforward to show that the spectral sequence collapses at the $r=1$: $E_1^{p,q}\simeq \delta_{(p,q),(0,0)}\C$. This completes the proof.
\endproof
Applying it to the tensor products of Weyl modules as $\lpp{-}\g$-modules, we obtain:
\begin{corollary}\label{homology vanishing}
For $\lambda,\mu\in P_+$ and $k,\ell\in \C$, we have  \\
\textup{(1)} $H_n(\lpp{-}\g,\weyl^k_\lambda)\simeq \delta_{n,0}L_\lambda$,\\
\textup{(2)} $H_n(\lpp{-}\g,\weyl^k_\lambda\otimes \weyl^\ell_\mu)=0$ $(n>0)$.
\end{corollary}
Next, we will show the analogous statement for the cohomology.
\begin{lemma}
For $\lambda\in P_+$ and $k\in \C\backslash \Q$, $\weyl^k_\lambda$ is simple as a $V^k(\g)$-module.
\end{lemma}
\proof
The proof is the same as the purely even case and thus we only sketch it.
Note that $L_0=\int T(z)zdz$ acts on $\weyl^k_\lambda$ semisimply and the $L_0$-eigenspace decomposition is $\weyl_{\lambda}^k=\bigoplus_{n\in \Z_{\geq0}} \weyl_{\lambda,\Delta_\lambda+n}^k$ with $\Delta_\lambda=\frac{(\lambda|\lambda+2\rho)}{2(k+h^\vee)}$.
Each $\weyl_{\lambda,\Delta_\lambda+n}^k$ is finite dimensional and $\weyl_{\lambda,\Delta_\lambda}^k=L_\lambda$.
Suppose that $\weyl^k_\lambda$ has a proper submodule $N$. Then $N$ must have a highest weight vector $v$ and gives rise to a non-zero homomorphism $\weyl_\mu^k\rightarrow \weyl_\lambda^k$ for some $\mu\in P_+$. 
By considering the $L_0$-eigenvalue of $v$, we have $\Delta_\mu=\Delta_\lambda+n$ for some $n\geq1$, which contradicts $k\in \C\backslash \Q$. This completes the proof.
\endproof
To apply Proposition \ref{condition for pairing}, we use the anti-isomorphism of Lie superalgebras
\begin{align*}
\begin{array}{cccc}
\tran\colon &\g[t^{\pm1}]\oplus \C K & \rightarrow & \g[t^{\pm1}]\oplus \C K\\
& X_{n}&\mapsto & (\iota X)_{-n}\\
& K& \mapsto & K
\end{array}
\end{align*}
where $X_n=Xt^n$ and  $\iota$ is the Chevalley anti-involution of $\g$.
It is obvious that $\tran$ restricts to $\tran\colon \lpp{-}\g\xrightarrow{\simeq} \lpp{+}\g$.
Let $\mathbb{M}^k_\lambda$ denote the Verma module of $\widehat{\g}$ generated by a highest weight vector $v_{\lambda,k}$ of highest weight $\lambda+k \Lambda_0$ and $\mathbb{M}^k_\lambda\twoheadrightarrow \weyl^k_\lambda$ the natural projection. Recall that $\mathbb{M}^k_\lambda$ has a unique bilinear form
$$\Psi_{\lambda,k}\colon\mathbb{M}^k_\lambda\otimes \mathbb{M}^k_\lambda \rightarrow \C,\quad v_{\lambda,k}\otimes v_{\lambda,k}\rightarrow 1$$
which satisfies (P2) in Proposition \ref{condition for pairing}. 
It induces a non-degenerate bilinear form on the simple quotient 
$\Psi_{\lambda,k}\colon \weyl^k_\lambda\otimes \weyl^k_\lambda \rightarrow \C$
for $k\in \C\backslash \Q$.
\begin{proposition}\label{cohomology vanishing}
For $\lambda,\mu\in P_+$ and $k,\ell\in \C\backslash \Q$, we have  \\
\textup{(1)} $H^n(\lpp{+}\g,\weyl^k_\lambda)\simeq \delta_{n,0}L_\lambda$,\\
\textup{(2)} $H^n(\lpp{+}\g,\weyl^k_\lambda\otimes \weyl^\ell_\mu)=0$ $(n>0)$.
\end{proposition}
\proof
 (1) is a direct consequence of Proposition \ref{condition for pairing} and Corollary \ref{homology vanishing}. (2) is also immediate by using the non-degenerate bilinear form
$$\Psi_{\lambda,k}\otimes \Psi_{\mu,\ell}\colon \left(\weyl^k_\lambda\otimes \weyl^\ell_\mu\right)\otimes  \left(\weyl^k_\lambda\otimes \weyl^\ell_\mu\right) \rightarrow \C.$$

\proof[Proof of Theorem \ref{main theorem of relcoh}]
It follows from Proposition \ref{spectral seq for relcoh}, Corollary \ref{homology vanishing} (2) and Proposition \ref{cohomology vanishing} (2) that $\relsemicoh{n}(\g,\mathbb{V}^k_\lambda\otimes \mathbb{V}^\ell_\mu)=0$ ($n\neq0$).
Then, the character of $\relsemicoh{0}(\g,\mathbb{V}^k_\lambda\otimes \mathbb{V}^\ell_\mu)$ is calculated by the Euler--Poincar\'{e} principle \cite{CFL}.
Let $\chi_{L_\lambda}$ denote the character of $L_\lambda$ as a $\g$-module, namely, 
$$\chi_{L_\lambda}(z)=\frac{\sum_{w\in W}(-1)^{\ell(w)}e^{w(\lambda+\rho)-\rho}}{\prod_{\alpha\in\Delta_{+}}(1-(-1)^{\bar{\alpha}}e^{-\alpha})^{\bar{\alpha}}}$$
and set
$$\Pi(z,q):=\prod_{n=1}^\infty(1-q^n)^{\mathrm{rank}\g}\prod_{\alpha\in \Delta_+}(1-(-1)^{\bar{\alpha}}z^\alpha q^n)^{\bar{\alpha}}(1-(-1)^{\bar{\alpha}}z^{-\alpha}q^n)^{\bar{\alpha}}.$$
It follows that 
\begin{align*}
\mathrm{ch}[\weyl^k_\lambda](z,q)=\frac{q^{\frac{(\lambda|\lambda+2\rho)}{2(k+h^\vee)}}\chi_{L_\lambda}(z)}{\Pi(z,q)},\quad 
\mathrm{sch}\left[\relsemi\right](z,q)=\Pi(z,q)^2
\end{align*}
and thus
\begin{align}\label{EP}
\mathrm{sch}\left[\weyl^k_\lambda\otimes \weyl^\ell_\mu\otimes \relsemi\right](z,q)=q^{\frac{|\lambda+\rho|^2-|\mu+\rho|^2}{2(k+h^\vee)}} \chi_{L_\lambda\otimes L_\mu}(z).
\end{align}
Here we have used the relation $k+h^\vee=-(\ell+h^\vee)$.
Since $C^{\frac{\infty}{2}+\bullet}_{\mathrm{rel}}(\g,\weyl^k_\lambda\otimes \weyl^\ell_\mu)=(\weyl^k_\lambda\otimes \weyl^\ell_\mu\otimes \relsemi)^\g$, we have 
$$\mathrm{ch}\left[\relsemicoh{0}(\g,\mathbb{V}^k_\lambda\otimes \mathbb{V}^\ell_\mu)\right](z,q)
=\mathrm{sch}\left[(\weyl^k_\lambda\otimes \weyl^\ell_\mu\otimes \relsemi)^\g\right](z,q).$$
By \eqref{EP}, it follows from the complete reducibility of integrable $\g$-modules that the right-hand side is  
\begin{align*}
\dim \Hom_\g(L_0,L_\lambda\otimes L_\mu) q^{\frac{|\lambda+\rho|^2-|\mu+\rho|^2}{2(k+h^\vee)}}=\delta_{\lambda,\mu^\dagger}q^{\frac{|\lambda+\rho|^2-|\mu+\rho|^2}{2(k+h^\vee)}}.
\end{align*}
Therefore, $\relsemicoh{0}(\g,\mathbb{V}^k_\lambda\otimes \mathbb{V}^\ell_\mu)$ is of one dimension and of the lowest conformal weight of the complex. The corresponding subcomplex is $(L_\lambda\otimes L_\mu)^\g$ with trivial differential, and thus the cohomology is $(L_\lambda\otimes L_\mu)^\g$ itself. Therefore, the assertion follows from Lemma \ref{multi of trivial rep}.
\endproof

\end{document}